\documentclass[graybox]{svmult}

\usepackage{type1cm}        
\usepackage{makeidx}   
\usepackage{graphicx}  
\usepackage{multicol} 
\usepackage[bottom]{footmisc}

\usepackage{newtxtext}       
\usepackage{newtxmath}       
\usepackage{subfigure}

\usepackage{url}

\usepackage{tikz}
\usetikzlibrary{arrows}
\usetikzlibrary{positioning}

\tikzstyle{block} = [rectangle, draw, text width=4em, text centered, rounded corners, minimum height=2em]
\tikzstyle{line} = [draw, -latex']

\def\R{\mathbb{R}}

\makeindex   

\begin{document}

\title*{Optimal control of vaccination and plasma transfusion with potential usefulness for COVID-19\thanks{%
This is a preprint of a paper whose final and definite form is published in
\emph{Analysis of Infectious Disease Problems (Covid-19) and Their Global Impact},
Springer Nature Singapore Pte Ltd.\\ 
Submitted 15/July/2020; revised 02/Oct/2020; accepted 08/Oct/2020.}}

\titlerunning{Optimal control of vaccination and plasma transfusion}

\author{Juliana Couras, Iv\'an Area, Juan J. Nieto, Cristiana J. Silva and Delfim F. M. Torres}

\authorrunning{J. Couras, I. Area, J. J. Nieto, C. J. Silva and D. F. M. Torres} 

\institute{Juliana Couras
\at University of Aveiro, 
3810-193 Aveiro, Portugal, 
\email{julianacouras@ua.pt}
\and 
Iv\'an Area 
\at Departamento de Matematica Aplicada II, 
E. E. Aeronaautica e do Espazo, Campus de Ourense, 
Universidade de Vigo, 32004 Ourense, Spain 
\email{area@uvigo.es}
\and 
Juan J. Nieto 
\at Instituto de Matematicas, Universidade de Santiago de Compostela, 
15782 Santiago de Compostela, Spain, 
\email{juanjose.nieto.roig@usc.es}
\and 
Cristiana J. Silva 
\at Center for Research and Development in Mathematics and Applications (CIDMA), 
Department of Mathematics, University of Aveiro, 
3810-193 Aveiro, Portugal, 
\email{cjoaosilva@ua.pt}
\and 
Delfim F. M. Torres 
\at 
Center for Research and Development in Mathematics and Applications (CIDMA), 
Department of Mathematics, University of Aveiro, 3810-193 Aveiro, Portugal, 
\email{delfim@ua.pt}}


\maketitle

\abstract*{The SEIR model is a compartmental model used to 
simulate the dynamics of an epidemic. In this chapter, 
we introduce two control functions in the compartmental SEIR model 
representing vaccination and plasma transfusion. Optimal control problems 
are proposed to study the effects of these two control measures, 
on the reduction of infected individuals and increase of recovered 
ones, with minimal costs. Up to our knowledge, the plasma transfusion 
treatment has never been considered as a control strategy for epidemics mitigation. 
The proposed vaccination and treatment strategies may have a real application 
in the challenging and hard problem of controlling the COVID-19 pandemic.}

\abstract{The SEIR model is a compartmental model used to 
simulate the dynamics of an epidemic. In this chapter, 
we introduce two control functions in the compartmental SEIR model 
representing vaccination and plasma transfusion. Optimal control problems 
are proposed to study the effects of these two control measures, 
on the reduction of infected individuals and increase of recovered 
ones, with minimal costs. Up to our knowledge, the plasma transfusion 
treatment has never been considered as a control strategy for epidemics mitigation. 
The proposed vaccination and treatment strategies may have a real application 
in the challenging and hard problem of controlling the COVID-19 pandemic.}


\section{Introduction}
\label{sec:Int}

Like many other physical and biological processes, epidemics can 
be modelled mathematically. Epidemic mathematical modelling is important, 
not only to understand the disease progression, but also to provide 
predictions about the epidemics evolution and insights about the dynamics 
of the transmission rate and the effectiveness of control measures. 
There are several compartmental models in epidemiology, like the 
$SI$, $SIR$, $SICA$ and the $SEIR$ model, see, e.g., 
\cite{brauer2011mathematical,Murray:book,SICA:2020} and references cited therein. 
In this chapter, we consider the SEIR model, where the human population 
is divided into four mutually exclusive compartments: susceptible $S$, 
latent $E$, infected $I$, and a recovered or removed (dead) $R$. 
We assume that the population is homogeneous and the various classes are uniformly mixed. 
We consider the case of constant total population $N$, that is, 
$S(t) + E(t) + I(t) + R(t) = N$ for every time $t$ in the time window $t \in [0, T]$
under study. In this case, the fraction of individuals in each compartment 
is defined as $s = S/N$, $e = E/N$, $i = I/N$ and $r = R/N$. The balance condition 
becomes $s + e + i + r = 1$. The assumptions made about the transmission of the infection 
and incubation period are reflected in the equations and parameters \cite{Murray:book} 
and are explained below. We consider the following parameters:
\begin{itemize}
\item transmission coefficient -- $\beta$;
\item infectious rate -- $\gamma$;
\item recovery rate -- $\mu$.
\end{itemize}
Then the $seir$ model is given by the following system of ordinary differential equations:  
\begin{align}
\begin{cases}
\frac{ds}{dt}(t) = -\beta \, s(t) \,i(t), \\
\frac{de}{dt}(t) = \beta \, s(t) \,i(t) - \gamma e(t), \\
\frac{di}{dt}(t) = \gamma \, e(t) - \mu \, i(t), \\
\frac{dr}{dt}(t) = \mu \, i(t) \, ,
\end{cases}
\label{eq:mod:SEIR}
\end{align}
represented graphically in the diagram of Figure~\ref{fig:diag:sir}.
\begin{figure}[ht!]
\centering
\begin{tikzpicture}
[node distance=2cm, auto]
\node[block] (s) {s};
\node[block, right=of s] (e) {e};
\node[block, right=of e] (i) {i};
\node[block, right=of i] (r) {r};
\path [line] (s) -- (e) node [midway,above] {$\beta$};
\path [line] (e) -- (i) node [midway,above] {$\gamma$};
\path [line] (i) -- (r) node [midway,above] {$\mu$} ;
\end{tikzpicture}
\caption{Diagram of the compartmental model \eqref{eq:mod:SEIR}.}
\label{fig:diag:sir}
\end{figure}
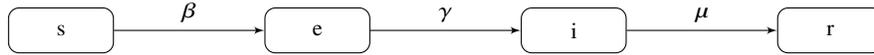
The term $\beta \, s \, i$ represents the gain in the exposed class, 
which is proportional to the fraction of infective (and infectious) 
and susceptible individuals, where the transmission coefficient $\beta > 0$ 
is a constant parameter. Individuals are transferred from the susceptible 
class $s$ to the exposed $e$ at this rate $\beta \, s \, i$. 
The incubation period is of $1/\gamma$ days, with $\gamma >0$, 
and after that time exposed individuals become infectious. The rate of removal 
of infective to the removed class is proportional to the number of infective, 
$\mu \, i$, with $\mu > 0$, where  $1/\mu$ is a measure of the time spent 
in the infectious state \cite{Murray:book}. 

The $seir$ model \eqref{eq:mod:SEIR} is an extension 
of the classical Kermack--McKendrick 1927 model  
\cite{KermackMcKendrick:1927,KermackMcKendrick:1991}, 
where the class of exposed (latent) individuals is considered.
SEIR type compartmental models have been extensively used to model 
the COVID-19 pandemic, see, e.g., \cite{covid19projectionsML,SEIR:comsol}, 
and researchers have shown that it can describe the spread of COVID-19 
in different countries: see \cite{Carcione:Lombardia} for a simulation 
of the COVID-19 spread in Lombardy (Italy) and also modifications 
of the SEIR model in \cite{Lopez:SEIR:Spain:Italy,Ng:COVID19,Prem:covid19:model}. 
Namely, in \cite{Lopez:SEIR:Spain:Italy} three classes are added for confined, 
under quarantine and COVID-19 induced deaths. The model in \cite{Ng:COVID19} 
considers the age of the population, time delay on the development of the pandemic, 
and resusceptibility to COVID-19 with temporal immune response. An age-structured 
SEIR model is proposed in \cite{Prem:covid19:model} considering 5-year bands 
until the age of 70 years and a single category aged 75 and older 
(resulting in 16 age categories for each class of individuals). 

Optimal control theory is a branch of mathematics that involves 
finding optimal ways of controlling a dynamic system \cite{Cesari,Pontryagin}. 
Optimal control has been applied to epidemiological models for many different 
infectious diseases, such as HIV/AIDS, malaria, Ebola, tuberculosis and cholera 
\cite{AreaEbola,LemosPaiao:cholera,HIVSilvaMaurer,TBHIVSilvaTorres}, 
and also non communicable diseases like cancer: see
\cite{Schattler:book} and references cited therein.  

Optimal control theory allows the study of the most cost-effective intervention 
strategy that changes the dynamics of a controlled system while minimizing 
a so-called objective function. In this chapter, we introduce two control 
functions in the $seir$ model \eqref{eq:mod:SEIR} that represent 
vaccination and plasma transfusion. Although vaccination has been widely studied 
from an optimal control point of view applied to epidemiological models, 
up to our knowledge, the plasma transfusion treatment has never been considered before. 
Plasma transfusion has been considered as a possible treatment for COVID-19, 
although it is still under study \cite{Plasma:2}. 

We propose five objective cost functionals and five corresponding optimal 
control problems for the three control systems that correspond to vaccination only, 
plasma transfusion treatment only, and combination of vaccination and plasma transfusion. 

This chapter is organized as follows. In Section~\ref{sec:2}, the vaccination and plasma 
transfusion are introduced in the $seir$ model, isolated and in combination, giving 
rise to three control systems that will be part of the optimal control problems 
proposed in Section~\ref{sec:optcontrol}. In Section~\ref{sec:num:results},
the solutions of the optimal control problems are compared numerically.
We end with Section~\ref{sec:conc} of discussion and conclusions.


\section{Control system: $seir$ model with vaccination and plasma transfusion}
\label{sec:2}

In this section, in order to control the spread of the infection, 
two types of interventions are introduced into the $seir$ 
model \eqref{eq:mod:SEIR}: vaccination $u$ and plasma transfusion $p$. 
Instead of representing the vaccination and plasma transfusion by constant 
positive parameters, we assume that vaccination and plasma transfusion are 
given by two functions $u(\cdot)$ and $p(\cdot)$, respectively, that change 
in time and that modify the dynamical behavior of model \eqref{eq:mod:SEIR}. 
In optimal control theory, functions $u(\cdot)$ and $p(\cdot)$ 
are usually called \emph{controls}. 

Starting by the vaccine, we introduce a control $u(\cdot)$ that represents 
the vaccination rate. By definition, it only makes sense to vaccinate people 
if they are susceptible to the disease. After being vaccinated, these people 
would become immune to the disease. In terms of the $seir$ model states, 
this means that an individual in the $s$ state would jump to the $r$ 
state after being vaccinated. Thus, the model must be rewritten in the following way:
\begin{align}
[\text{\emph{vaccination  based  control}}]\quad
\begin{cases}
\frac{ds}{dt}(t) = -\beta \, s(t) \,i(t) - u(t) \, s(t),\\
\frac{de}{dt}(t) = \beta \, s(t) \, i(t) - \gamma \, e(t), \\
\frac{di}{dt}(t) = \gamma \, e(t) - \mu \, i(t), \\
\frac{dr}{dt}(t) = \mu \, i(t) + u(t) \, s(t),  
\end{cases}
\label{eq:vaccines}
\end{align}
where the control function $u(\cdot)$ is bounded between 0 and $u_{\max} \leq 1$. 

Regarding treatment, the aim is to emulate a serological treatment, that is, 
a plasma transfusion. A plasma transfusion consists on infusing sick individuals 
with the blood plasma harvested from the immune individuals. Thus, in terms 
of the model, it requires that a recovered individual $r$ donates plasma 
to an infectious individual $i$. The control is the rate at which this 
transfusion happens. Let the control be $p(\cdot)$. Then, the $seir$ model 
\eqref{eq:mod:SEIR} is rewritten in the following way:
\begin{align}
[\text{\emph{plasma transfusion based control}}] \quad
\begin{cases}
\frac{ds}{dt}(t) = -\beta \, s(t) \, i(t), \\
\frac{de}{dt}(t) = \beta \, s(t) \, i(t) - \gamma \, e(t), \\
\frac{di}{dt}(t) = \gamma \, e(t) - \mu \, i(t) - p(t) \, r(t) \, i(t), \\
\frac{dr}{dt}(t) = \mu \, i(t) + p(t) \, r(t) \, i(t), 
\end{cases}
\label{eq:plasma}
\end{align}
where the control $p(\cdot)$ satisfies the control 
constraint $0 \leq p(\cdot) \leq p_{\max} \leq 1$. 

Finally, the two previous controls are considered simultaneously, 
being the resulting model the following:
\begin{align}
[\text{\emph{vaccination and plasma transfusion}}]\quad
\begin{cases}
\frac{ds}{dt}(t) = -\beta \, s(t) \, i(t) - u(t) \, s(t), \\
\frac{de}{dt}(t) = \beta \, s(t) \, i(t) - \gamma \,  e(t), \\
\frac{di}{dt}(t) = \gamma \,  e(t) - \mu \, i(t) - p(t) \, r(t) \, i(t), \\
\frac{dr}{dt}(t) = \mu \, i(t) + p(t) \, r(t) \, i(t) + u(t) \, s(t).
\end{cases}
\label{eq:vaccine-plasma}
\end{align}
The set of admissible controls functions is given by 
\begin{equation}
\label{eq:admcont}
\Omega = \left\{ \left(u(\cdot), p(\cdot) \right) \in \left( L^\infty(0, T) \right)^2 \, | \, 0 \leq u(t) 
\leq u_{\max} \, , 0 \leq p(t) \leq p_{\max} \, ,  \, \forall \, t \in [0, T]   \right\} \, .
\end{equation}


\section{Optimal control}
\label{sec:optcontrol}

Consider non-negative initial conditions for the state variables 
$(s, e, i, r) \in \left(R_0^+ \right)^4$: 
\begin{equation}
\label{eq:init}
s(0) \geq 0 \, , \quad e(0) \geq 0 \, , \quad i(0) \geq 0 \, , \quad r(0) \geq 0 \, ,
\end{equation}
where the state variables satisfy 
$s(t) + e(t) + i(t) + r(t) = 1$, for all $t \in [0, T]$. 
In order to formulate an optimal control problem, a cost functional 
needs to be proposed, which in our case we intend to maximize. 
We propose an optimal control problem for the control systems 
given by \eqref{eq:vaccines}, \eqref{eq:plasma}
or \eqref{eq:vaccine-plasma}, 
with five different $L^2$ objective functionals, denoted 
for simplicity by $J_i$, $i = 1, \ldots, 5$.
All of them are obtained from
$$
\mathcal{J}_\eta
= \int_0^T \left( \eta_1 r(t) - \eta_2 i(t)  - \eta_3 u^2(t) - \eta_4 p^2(t) \right) dt 
$$
as follows: $J_1 = \mathcal{J}_{(0,1,1,0)}$,
$J_2 = \mathcal{J}_{(1,1,1,0)}$,
$J_3 = \mathcal{J}_{(0,1,0,1)}$,
$J_4 = \mathcal{J}_{(1,1,0,1)}$,
and $J_5 = \mathcal{J}_{(0,1,1,1)}$.
Other cases of cost functionals are obviously possible,
but we found these five to be the most interesting.
We also do not consider all possible combinations between
the three control systems and the five costs to be maximized,
restricting ourselves to five optimal control problems.
Regarding the vaccination based control \eqref{eq:vaccines}, 
we consider the two objective functionals
\begin{equation}
\label{eq:J1}
J_1(u(\cdot)) = \int_0^T \left( - i(t) - u^2(t) \right) dt
\end{equation}
and
\begin{equation}
\label{eq:J2}
J_2(u(\cdot)) = \int_0^T \left( r(t) - i(t) - u^2(t) \right) dt \, .
\end{equation}
When the cost functional is considered to be $J_1$, 
the main goal of maximizing the functional is to minimize the fraction 
of infected individuals and, at the same time, the vaccination costs. We compare 
the solution to this optimal control problem with the one that maximizes $J_2$, 
that is, the one that maximize the fraction of recovered (immune) individuals and, 
simultaneously, minimizes the fraction of infected individuals and the vaccination costs. 
The numerical solutions are compared in Section~\ref{sec:num:results}. 

When only the treatment by plasma transfusion is considered, 
that is, when we focus ourselves on the control system \eqref{eq:plasma},
we use the objective functionals $J_3$ and $J_4$: 
\begin{equation}
\label{eq:J3}
J_3(p(\cdot)) = \int_0^T \left( - i(t) - p^2(t) \right) dt \, , 
\end{equation}
\begin{equation}
\label{eq:J4}
J_4(p(\cdot)) = \int_0^T \left( r(t) - i(t) - p^2(t) \right) dt \, , 
\end{equation}
where maximizing $J_3$ corresponds to minimizing the 
fraction of infected individuals and the costs associated with plasma 
transfusion treatment, and for maximizing $J_4$ the main goal 
is to maximize the fraction of recovered individuals, by treatment, and, 
at the same time, minimize the fraction of infected individuals 
with less treatment cost as possible. 

Finally, when both controls are considered simultaneously, modelled 
by the vaccination and plasma transfusion based control system 
\eqref{eq:vaccine-plasma}, the objective functional 
considered to be maximized was $J_5$: 
\begin{equation}
\label{eq:J5}
J_5(u(\cdot), p(\cdot)) = \int_0^T \left( - i(t) - u^2(t) - p^2(t) \right) dt 
\end{equation}
with the main goal to minimize the fraction of infected individuals and the costs 
associated with vaccination and plasma transfusion treatment. 

Associated to each of the cost functionals $J_i$, $i = 1, \ldots, 5$, 
we propose an optimal control problem of determining the state trajectories 
$\left(s^*(\cdot), e^*(\cdot), i^*(\cdot), r^*(\cdot)\right)$, associated to an admissible 
control $u^*(\cdot) \in \Omega$ and/or $p^*(\cdot) \in \Omega$ on the time interval $[0, T]$, 
satisfying one of the control systems \eqref{eq:vaccines}--\eqref{eq:vaccine-plasma}, 
as explained, the initial conditions \eqref{eq:init}, and maximizing the 
corresponding functional. 
The five optimal control problems are denoted by $(OC_i)$, $i = 1, \ldots 5$,
and are now summarized.
\paragraph{Vaccination based control system \eqref{eq:vaccines} and maximizing the cost functional $J_1$:}
\begin{equation}
\label{eq:cost:max:1}
\tag{$OC_1$}
J_1(u^*(\cdot)) = \max_{\Omega} \int_0^T \left( - i(t) - u^2(t) \right) \, dt.
\end{equation}
\paragraph{Vaccination based control system \eqref{eq:vaccines} and maximizing the cost functional $J_2$:} 
\begin{equation}
\label{eq:cost:max:2}
\tag{$OC_2$}
J_2(u^*(\cdot)) = \max_{\Omega} \int_0^T \left( r(t) - i(t) - u^2(t) \right) dt.
\end{equation}
\paragraph{Plasma transfusion based control system \eqref{eq:plasma} and maximizing the cost functional $J_3$:} 
\begin{equation}
\label{eq:cost:max:3}
\tag{$OC_3$}
J_3(p^*(\cdot)) = \max_{\Omega} \int_0^T \left( - i(t) - p^2(t) \right) dt.
\end{equation}
\paragraph{Plasma transfusion based control system \eqref{eq:plasma} and maximizing the cost functional $J_4$:} 
\begin{equation}
\label{eq:cost:max:4}
\tag{$OC_4$}
J_4(p^*(\cdot)) = \max_{\Omega} \int_0^T \left( r(t) - i(t) - p^2(t) \right) dt.
\end{equation}
\paragraph{Vaccination and plasma transfusion control system 
\eqref{eq:vaccine-plasma} and maximizing the cost functional $J_5$:}
\begin{equation}
\label{eq:cost:max:5}
\tag{$OC_5$}
J_5(u^*(\cdot), p^*(\cdot)) = \max_{\Omega} \int_0^T \left( - i(t) - p^2(t) - u^2(t) \right) dt. 
\end{equation}
Note that all optimal control problems have a $L^2$-cost functional, 
in other words, the integrand of the cost $J_i$, $i=1, \ldots, 5$, 
is always convex with respect to the controls $u$ and $p$. 
Moreover, the control systems \eqref{eq:vaccines}--\eqref{eq:vaccine-plasma} 
are Lipschitz with respect to the state variables $\left(s, e, i, r \right)$.
These properties ensure the existence of an optimal control 
$\left( u^*(\cdot), p^*(\cdot) \right)$  
for the optimal control problems $(OC_1)$--$(OC_5)$. 
Moreover, we apply the Pontryagin Maximum Principle (see, e.g., \cite{Pontryagin}), 
which is a first order necessary optimality condition. The obtained result, 
here formulated and proved for the optimal control 
problem $(OC_1)$, can be trivially extended to the other optimal 
control problems $(OC_i)$, $i = 2, \ldots, 5$. 

\begin{theorem}
\label{theo:PMP}
The optimal control problem $(OC_1)$ with fixed final time $T$ 
admits a unique optimal solution $\left(s^*(\cdot), e^*(\cdot), 
i^*(\cdot), r^*(\cdot)  \right)$ associated to the optimal control 
$u^*(\cdot)$ given by 
\begin{equation}
\label{eq:opt:cont}
u^*(t)  =\min\left\{  \max\left\{  0, \frac{ (\lambda_4(t) 
- \lambda_1(t))\,s(t)}{2}  \right\}, u_{\max} \right\}
\end{equation}	
on $[0, T]$, where the adjoint functions $\lambda_i$ satisfy	 
\begin{eqnarray}
\label{eq:adjointsys}
\begin{cases}
\dot{\lambda}_1(t) = -\lambda_1(t)\, \left( -i(t)\,\beta- u(t) \right) 
-\lambda_2(t) \,i\,\beta - \lambda_4(t)\,u(t) , \\
\dot{\lambda}_2(t) =  \lambda_2(t) \,\gamma-\lambda_3(t)\,\gamma ,\\
\dot{\lambda}_3(t) = 1 + \lambda_1(t)\,\beta\,s(t)
-\lambda_2(t)\,\beta\,s(t)+\lambda_3(t)\,\mu-\lambda_4(t)\,\mu,\\
\dot{\lambda}_4(t) = 0, 
\end{cases}
\end{eqnarray}
with the transversality conditions $\lambda_i(T) = 0$, $i = 1, \ldots , 4$. 
\end{theorem}

\begin{proof}
The existence of an optimal control $u^*(\cdot)$ of the optimal control problem $(OC_1)$ 
is due to the convexity of the $L^2$-cost functional $J_1$ and to the fact 
that the vaccinated based control system \eqref{eq:vaccines} is Lipschitz with respect 
to the state variables $(s, e, i, r)$: see, e.g., \cite{Cesari}. 
The uniqueness of the optimal control $u^*$ comes from the boundedness of the state 
and adjoint functions and the Lipschitz property of system \eqref{eq:vaccines} 
(see \cite{Jung:2002,SilvaTorres:2013} and references cited therein).	
According to the Pontryagin Maximum Principle, if $u^*(\cdot)$ is optimal 
for the problem $(OC_1)$ with fixed final time $T$, then there exists a nontrivial 
absolutely continuous mapping $\Lambda : [0, T] \to \R^4$, $\Lambda(t) 
= \left( \lambda_1(t), \lambda_2(t), \lambda_3(t), \lambda_4(t) \right)$, 
called the \emph{adjoint vector}, such that
$$
\dot{s} = \frac{\partial H}{\partial \lambda_1}, \quad 
\dot{e} = \frac{\partial H}{\partial \lambda_2}, \quad 
\dot{i} = \frac{\partial H}{\partial \lambda_3}, \quad 
\dot{r} = \frac{\partial H}{\partial \lambda_4} 
$$
and
$$
\dot{\lambda}_1 =- \frac{\partial H}{\partial s}, \quad 
\dot{\lambda}_2 =- \frac{\partial H}{\partial e}, \quad 
\dot{\lambda}_3 =- \frac{\partial H}{\partial i}, \quad 
\dot{\lambda}_4 =- \frac{\partial H}{\partial r} \, , 
$$
where 
\begin{equation*}
\begin{split}
H &= H(s(t), i(t), c(t), a(t)) = - i(t) - u^2(t)\\ 
&+ \lambda_1(t) \left(  -\beta \, s(t) \,i(t) - u(t) \, s(t) \right)\\
&+ \lambda_2(t) \left(  \beta \, s(t) \,i(t) - \gamma \, e(t) \right)\\
&+ \lambda_3(t) \left( \gamma \, e(t) - \mu \, i(t)  \right)\\
&+ \lambda_4(t) \left(  \mu \, i(t) + u(t) \, s(t) \right) 
\end{split} 
\end{equation*}
is the \emph{Hamiltonian}, and the maximality condition 
\begin{multline*}
H(s^*(t), e^*(t), i^*(t), r^*(t), \lambda^*(t), u^*(t)) \\
= \max_{ 0 \leq u \leq u_{\max}} H(s^*(t), e^*(t), i^*(t), r^*(t), \lambda^*(t), u(t)) 
\end{multline*}
holds almost everywhere on $[0, T]$. Moreover, the transversality conditions
$$
\lambda_i(T) = 0 \, , \quad \quad i = 1, \ldots , 4,
$$
hold. Furthermore, from the maximality condition, we have
\begin{equation*}
u^*(t)  =\min\left\{  \max\left\{  0, \frac{ (\lambda_4(t) 
- \lambda_1(t))\,s(t)}{2}  \right\}, u_{\max} \right\} \, .
\end{equation*}
The proof is concluded.
\end{proof}

To solve optimal control problems numerically, two approaches are possible: direct and indirect. 
Indirect methods are based on Pontryagin's Maximum Principle but not very much 
widespread since they are not immediately available by software packages.
We refer the reader to \cite{CarlosCampos} for the implementation of Pontryagin's 
Maximum Principle using Octave/MATLAB.
Direct methods consist in the discretization of the optimal control problem, 
reducing it to a nonlinear programming problem \cite{[16],[15]}. In the next section, 
we use the Applied Modeling Programming Language AMPL \cite{AMPL} 
to discretize the optimal control problems $(OC_i)$, $i = 1, \ldots, 5$. 
Then, the resulting nonlinear programming problems are solved using the 
Interior-Point optimization solver developed by W\"achter and Biegler \cite{IPOPT}, 
through the NEOS Server \cite{neos}. For more details on the numerical 
aspects see \cite{MaurerSilvaTorresMBE}.


\section{Numerical simulations and results}
\label{sec:num:results}

In this section, we provide numerical simulations for the solutions of the optimal 
control problems $(OC_i)$, $i = 1, \ldots, 5$, proposed in Section~\ref{sec:optcontrol}. 
The following values for the initial conditions are considered: 
\begin{equation}
\label{eq:init:cond:value}
s(0) = 0.88 \, , \quad  e(0) = 0.07 \, , \quad 
i(0) = 0.05 \, , \quad
r(0) = 0 \, , 
\end{equation}
and the parameter values
\begin{equation}
\label{eq:param}
\beta = 0.3 
\, , \quad \gamma = 0.1887 \, , \quad
 \mu = 0.1 \, .
\end{equation}
The initial conditions \eqref{eq:init:cond:value} 
were chosen arbitrarily, considering an hypothetical situation where 88\% 
of the total population is susceptible to the disease and there is a relatively small percentage 
of infected population and no recovered individuals. The parameter values are chosen in such a way 
that model \eqref{eq:mod:SEIR} simulates an epidemic outbreak, caused by a communicable disease.
Moreover, we consider the control constraints with $u_{\max} = 0.5$ 
and $p_{\max} = 0.3$, that is, the admissible controls $(u, p) \in \Omega$ 
must satisfy $0 \leq u(t) \leq 0.5$ and $0 \leq p(t) \leq 0.3$ for all $t \in [0, T]$. 

All computations have been performed with an Intel i7-4720HQ 2.60GHz processor,
8 GB of RAM, and an SSD disk of 128 GB under Windows 10, Home Edition of 64 bits.


\subsection{The $seir$ model without controls}

The $seir$ model differential equations were integrated using the \textsf{ode45} 
MATLAB routine, which is based on an explicit Runge--Kutta method \cite{matlab:ode}.
For the parameter values \eqref{eq:param}, the dynamic evolution of the uncontrolled 
system \eqref{eq:mod:SEIR} is described in 
Figures~\ref{fig:seir-no-control}--\ref{fig:individual-seir-no-control}. 
We observe that although the fraction of infected $i$ and exposed $e$ individuals tends 
to 0, after 100 units of time, more than 40 per cent of the population recovered 
or was removed (possible died from the disease). 
\begin{figure}[ht!]
\centering
\includegraphics[scale=0.4]{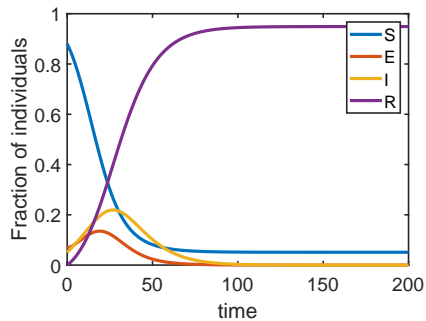}
\caption{Joint evolution of the state variables $s$, $e$, $i$ and $r$ 
of the uncontrolled model \eqref{eq:mod:SEIR}, during 200 time units.}
\label{fig:seir-no-control}
\end{figure}
\begin{figure}[ht!]
\begin{center}
\subfigure[]{\includegraphics[scale=0.3]{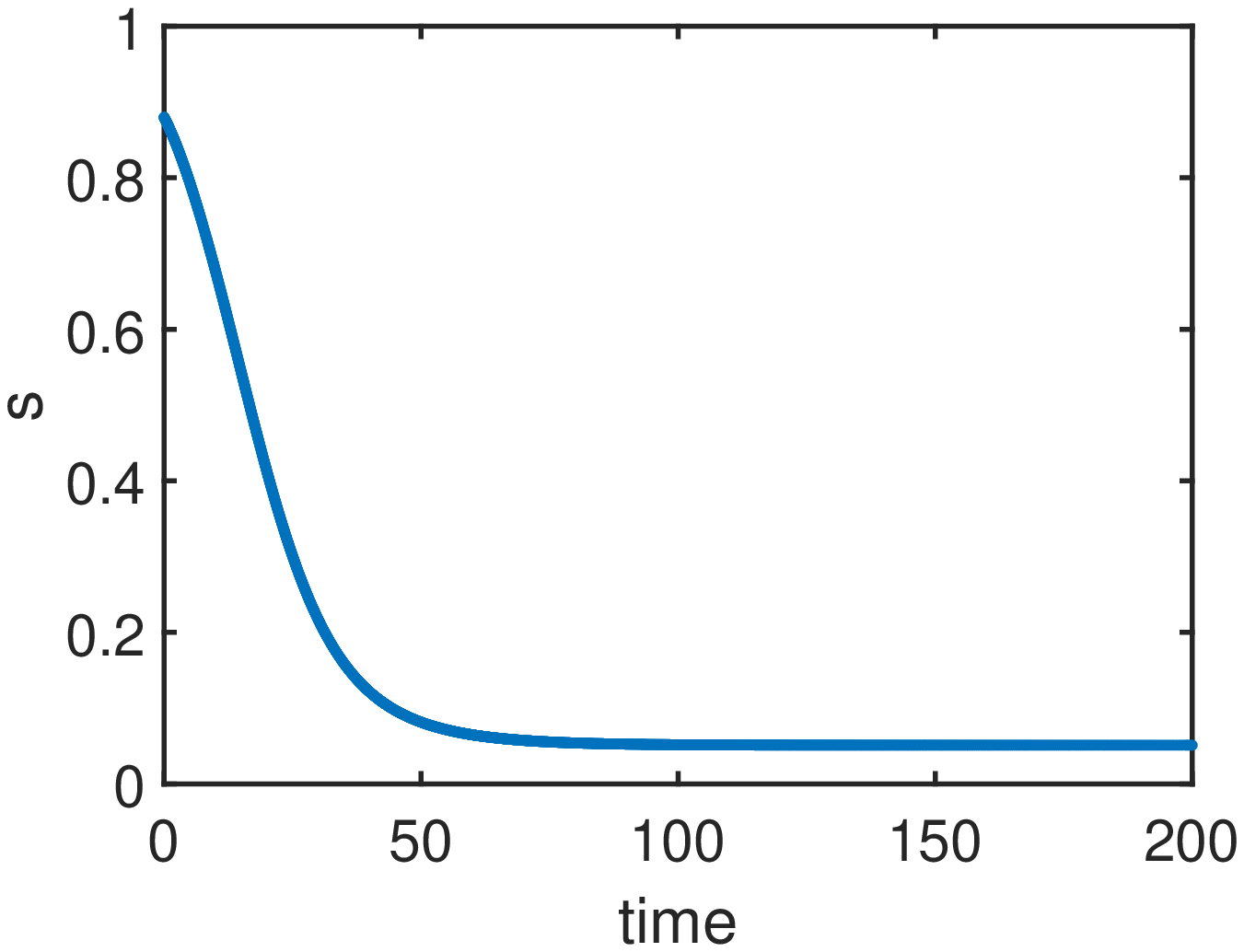}}
\quad
\subfigure[]{\includegraphics[scale=0.3]{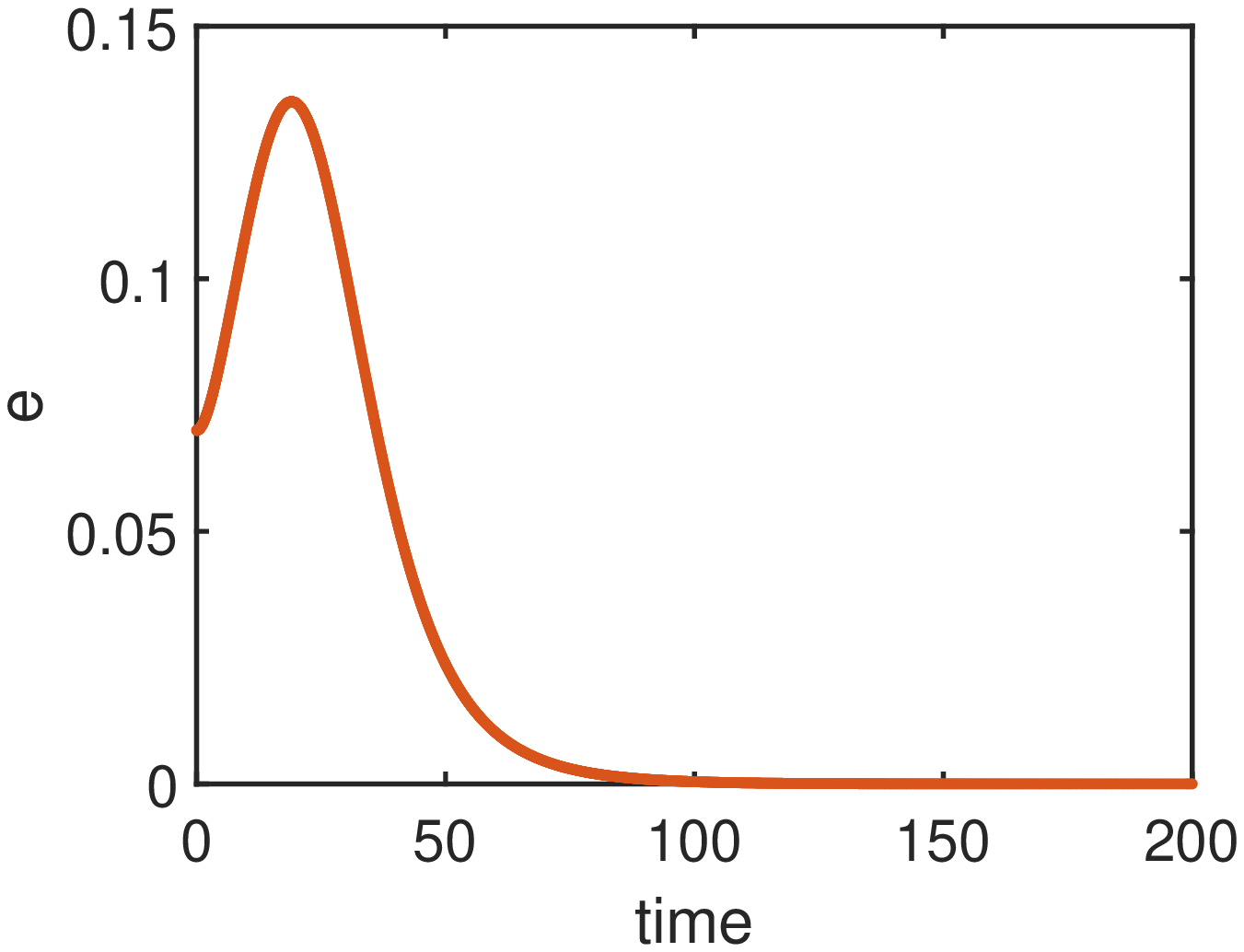}}\\
\subfigure[]{\includegraphics[scale=0.3]{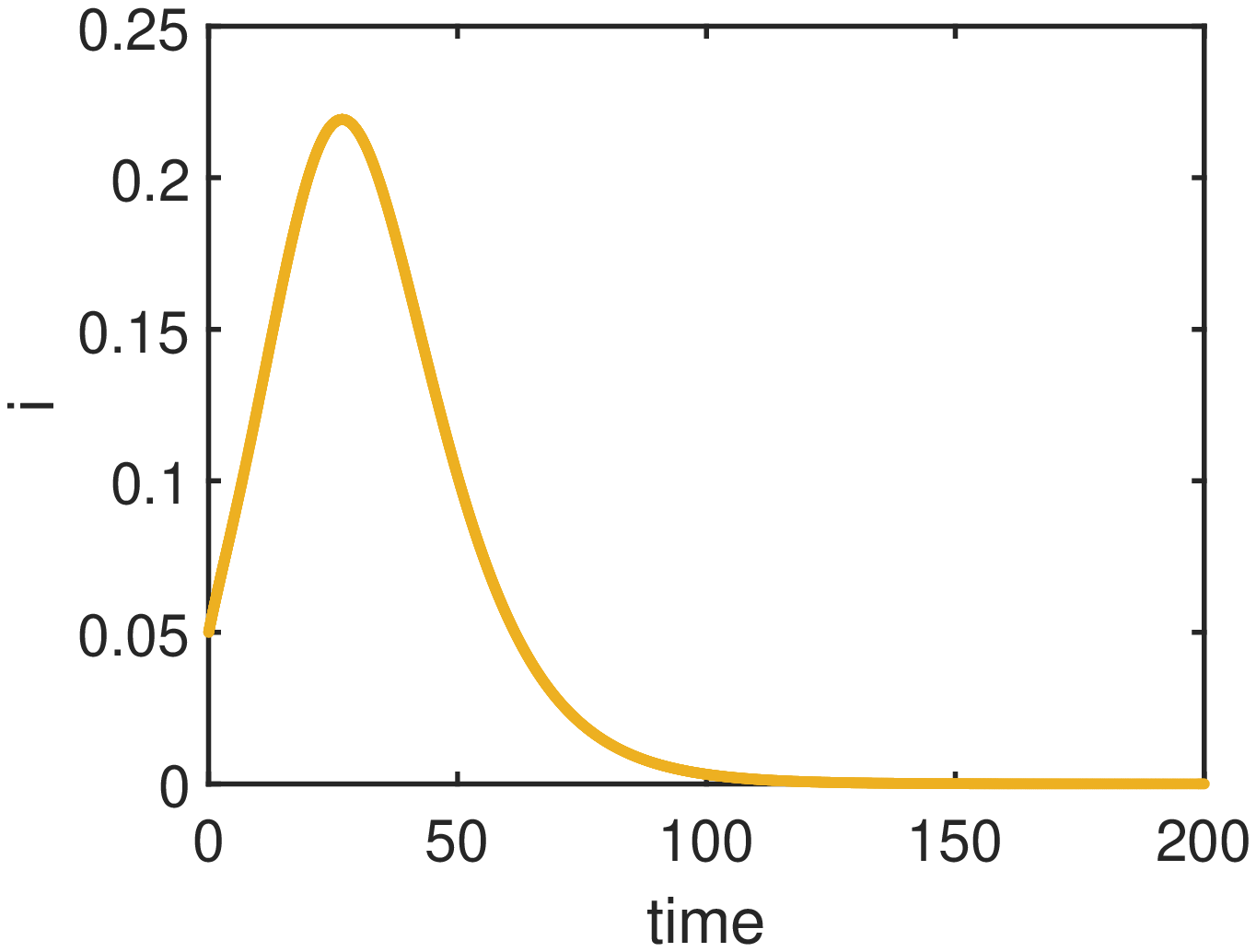}}
\quad
\subfigure[]{\includegraphics[scale=0.3]{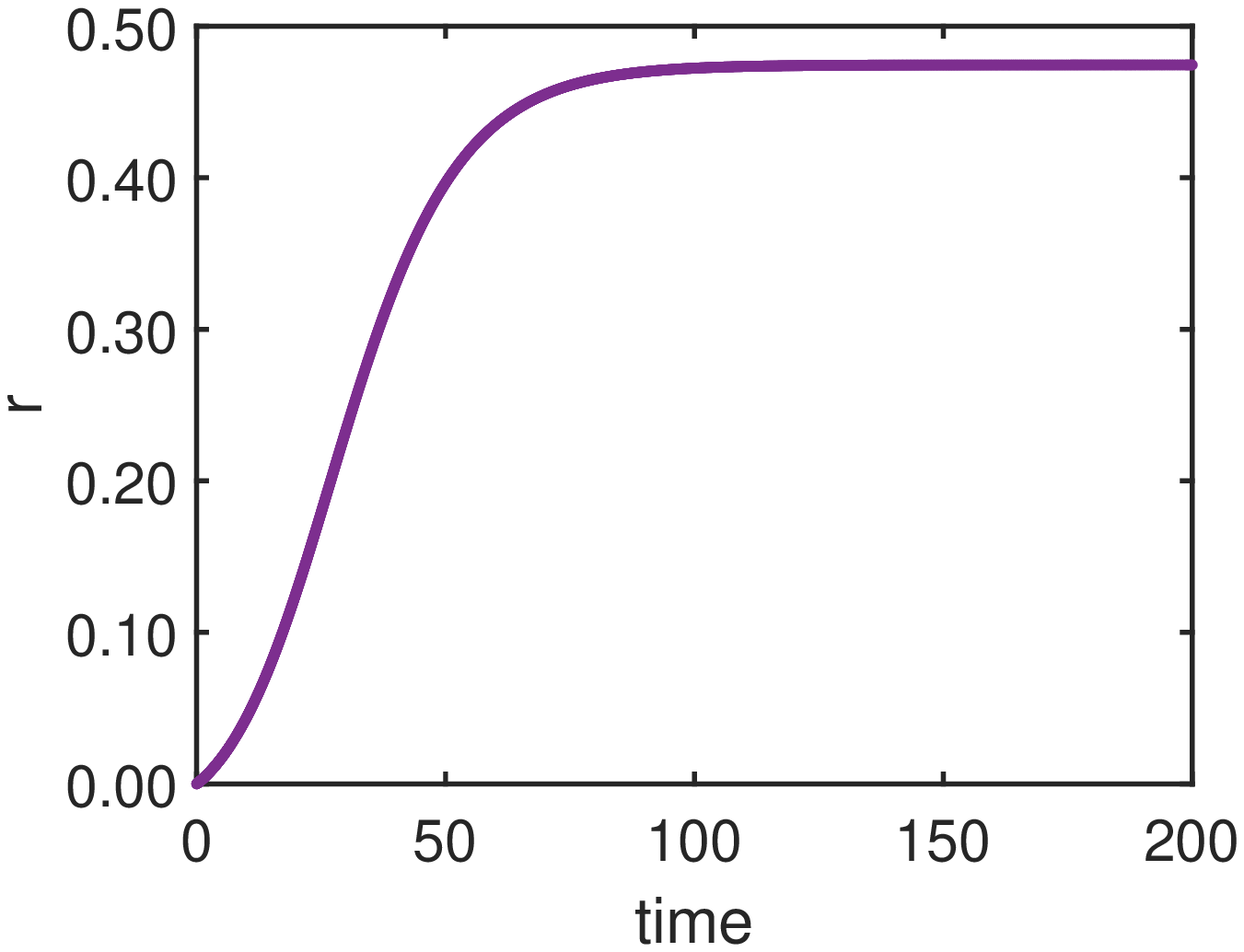}}
\caption{Individual evolution of the state variables $s e i r$ 
of the uncontrolled model \eqref{eq:mod:SEIR}, during 200 time units. 
(a) susceptible state $s$. (b) exposed state $e$. 
(c) infected state $i$. (d) recovered state $r$.}
\label{fig:individual-seir-no-control}
\end{center}
\end{figure}


These results were obtained in ``real time'' under MATLAB.


\subsection{The $seir$ model with controls}

It is desirable to minimize the fraction of infected individuals 
that get infected by the disease with minimal costs. 

\subsubsection{Optimal control problems $(OC_1)$ and $(OC_2)$}

Firstly, we consider the effect of vaccinating the population 
at the first 20 time units, aiming at maximizing $J_1$ subject 
to the vaccination based control system \eqref{eq:vaccines}, 
the initial conditions \eqref{eq:init:cond:value}, and the control constraint 
$0 \leq u(t) \leq 0.5$. The results obtained are given
in Figures~\ref{fig:vaccine} and \ref{fig:vaccinev2}.
\begin{figure}[ht!]
\begin{center}
\subfigure[]{\includegraphics[scale=0.3]{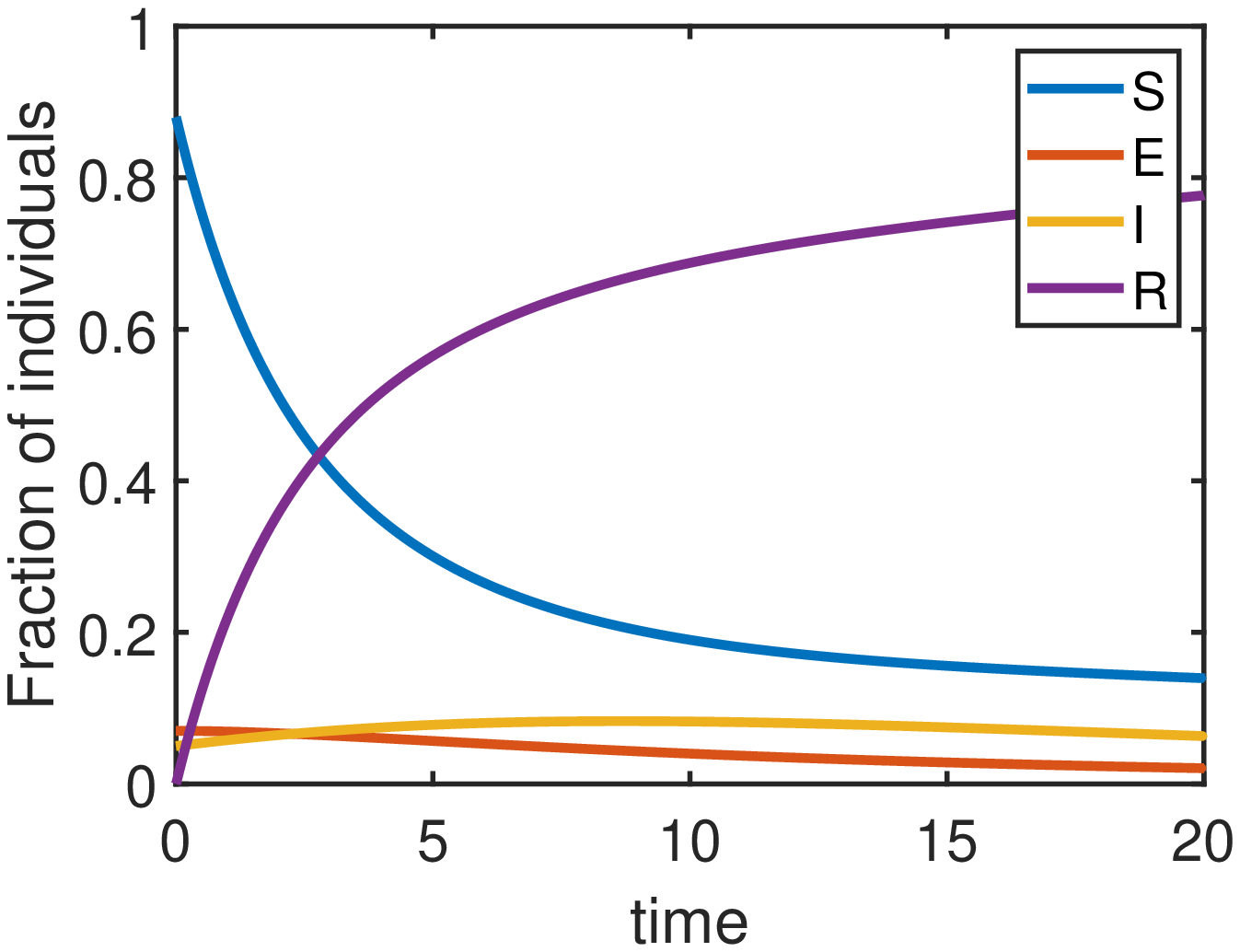}}
\quad
\subfigure[]{\includegraphics[scale=0.3]{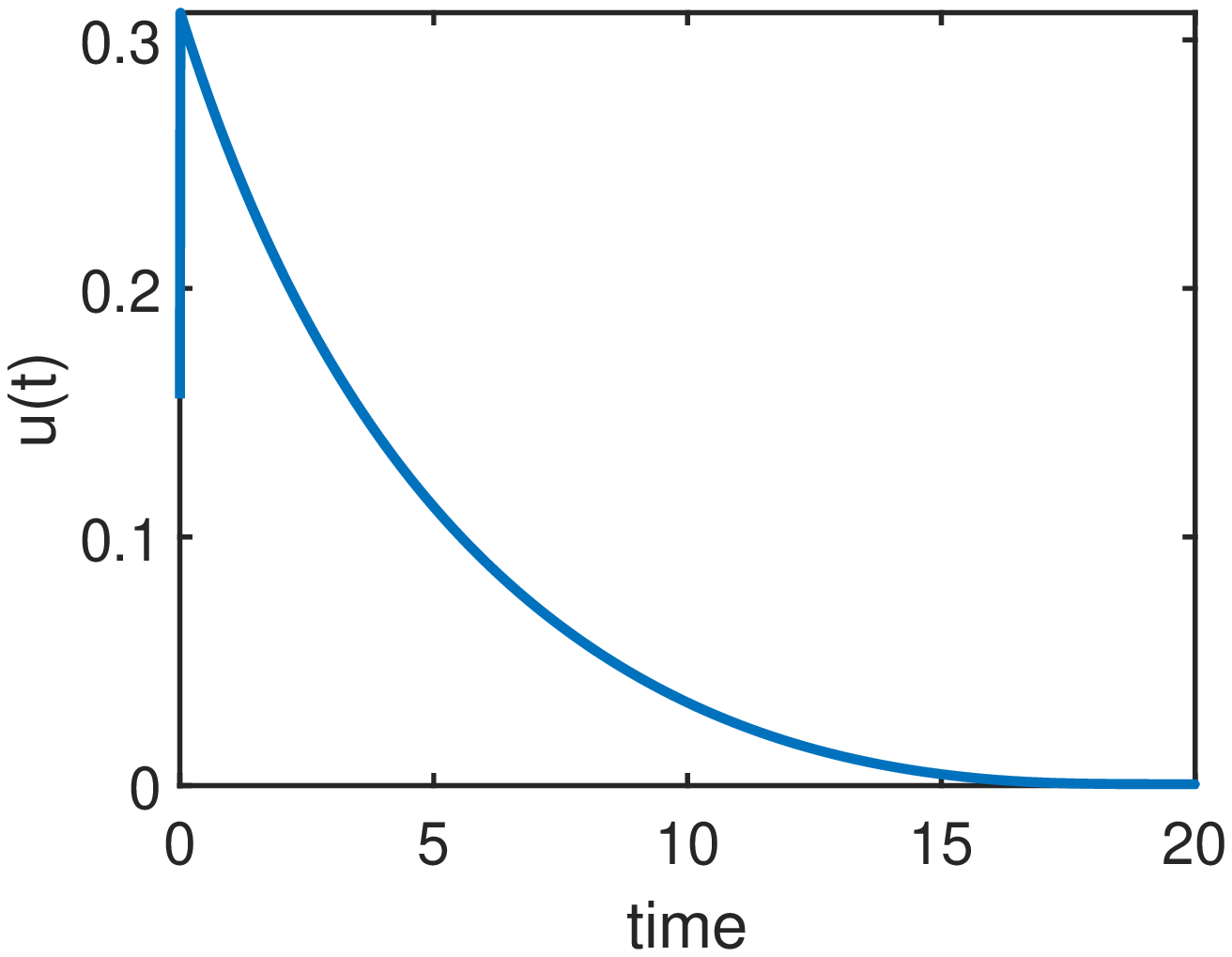}}\\
\caption{Effect of vaccinating the population during 20 time units 
considering $(OC_1)$. (a) $seir$ state variables applying vaccination. 
(b) Vaccination control $u(\cdot)$. }
\label{fig:vaccine}
\end{center}
\end{figure}
\begin{figure}[ht!]
\begin{center}
\subfigure[]{\includegraphics[scale=0.3]{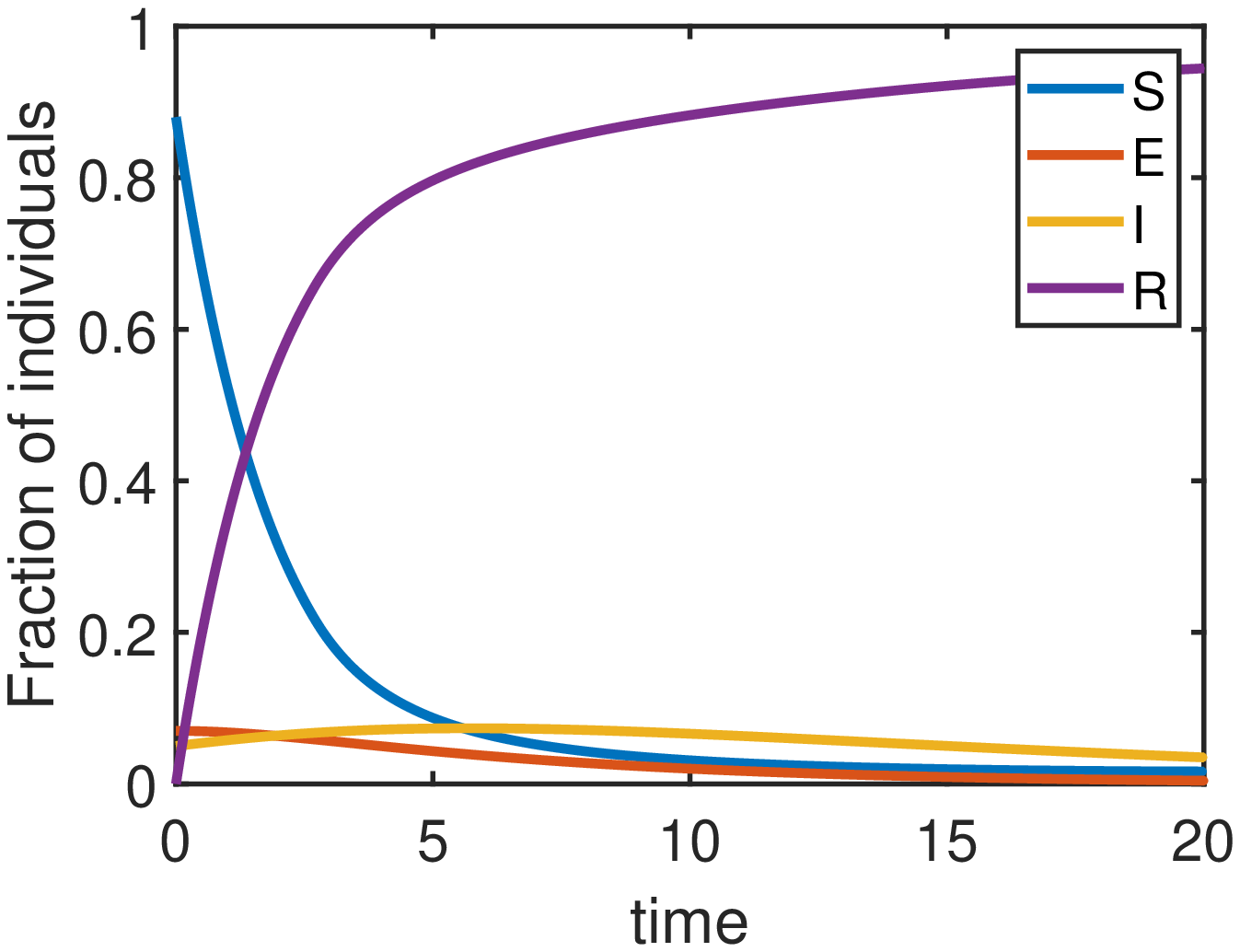}}
\quad
\subfigure[]{\includegraphics[scale=0.3]{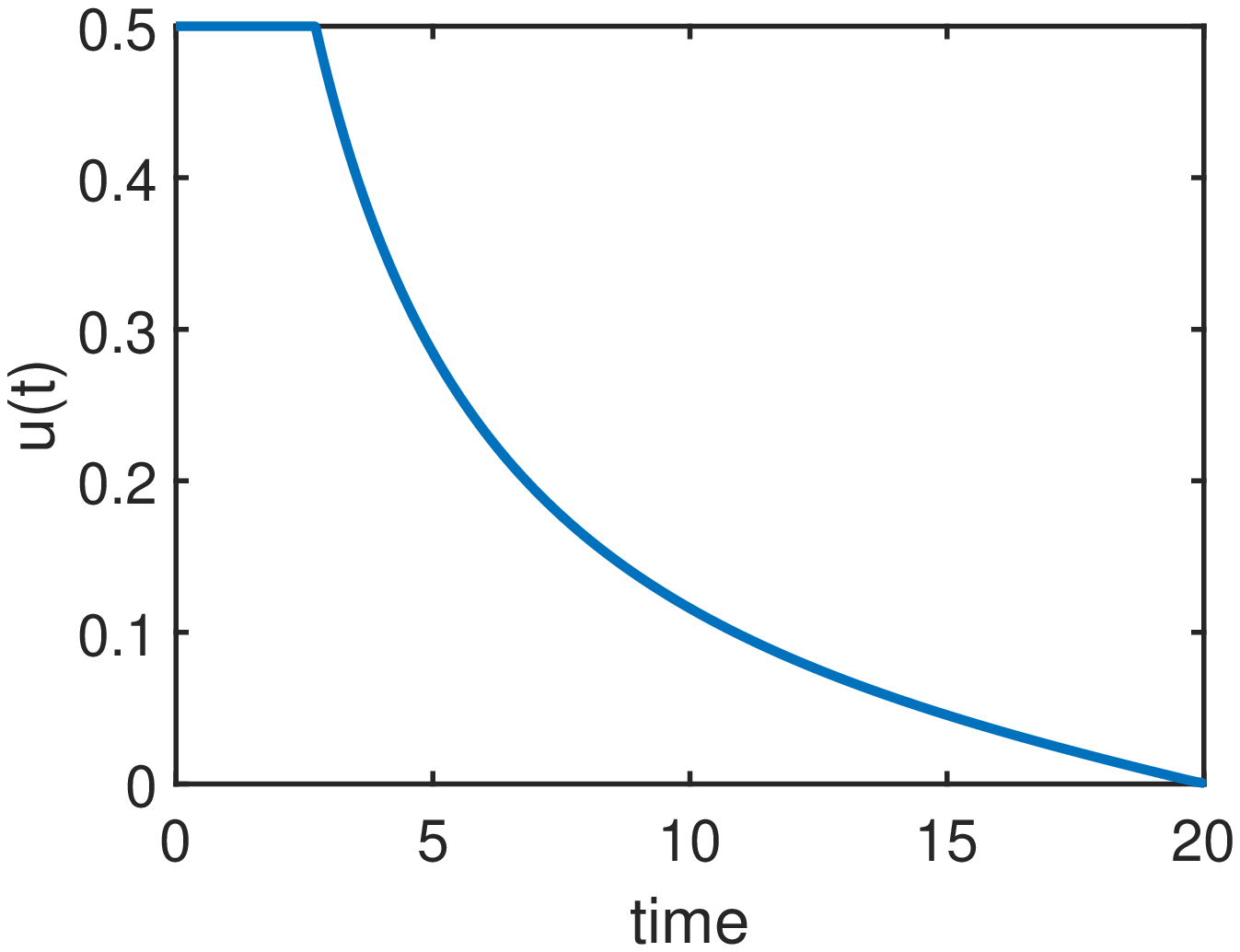}}\\
\caption{Effect of vaccinating the population during 20 time units 
considering $(OC_2)$. (a) $seir$ state variables applying vaccination. 
(b) Vaccination control $u(\cdot)$. }
\label{fig:vaccinev2}
\end{center}
\end{figure}

We see that in both $seir$ evolutions, the susceptible $s$ and recovered $r$ states 
seem to interchange, as expected by the vaccination based control system 
\eqref{eq:vaccines}. Further, looking at the vaccination control $u(\cdot)$ evolution, 
it is possible to see that in both Figures~\ref{fig:vaccine} and \ref{fig:vaccinev2} 
its value starts at a maximum and then decays as time passes. This makes sense, 
since at the beginning of the epidemic there are more susceptible individuals $s$. 
Thus, it is expected that the rate of vaccination is larger at this time in order 
to try to vaccinate the most susceptible individuals $s$ as possible before they start 
getting infected. Further, comparing the vaccination control of Figure~\ref{fig:vaccine} 
and Figure~\ref{fig:vaccinev2}, one can see that applying the condition of maximizing 
the fraction of recovered individuals $r$ with the cost functional $J_2$ translates into keeping 
the rate of vaccination at its maximum for 3 units of time before starting to decay 
with a less steeper slope than its analogue in $J_1$. 


\subsubsection{Optimal control problems $(OC_3)$ and $(OC_4)$}

Regarding the plasma transfusion treatment, one can see that, in contrast 
to the vaccine control, here the control $p(\cdot)$ peaks later in time
(see Figures~\ref{fig:plasma} and \ref{fig:plasmav2}). 
\begin{figure}[ht!]
\begin{center}
\subfigure[]{\includegraphics[scale=0.3]{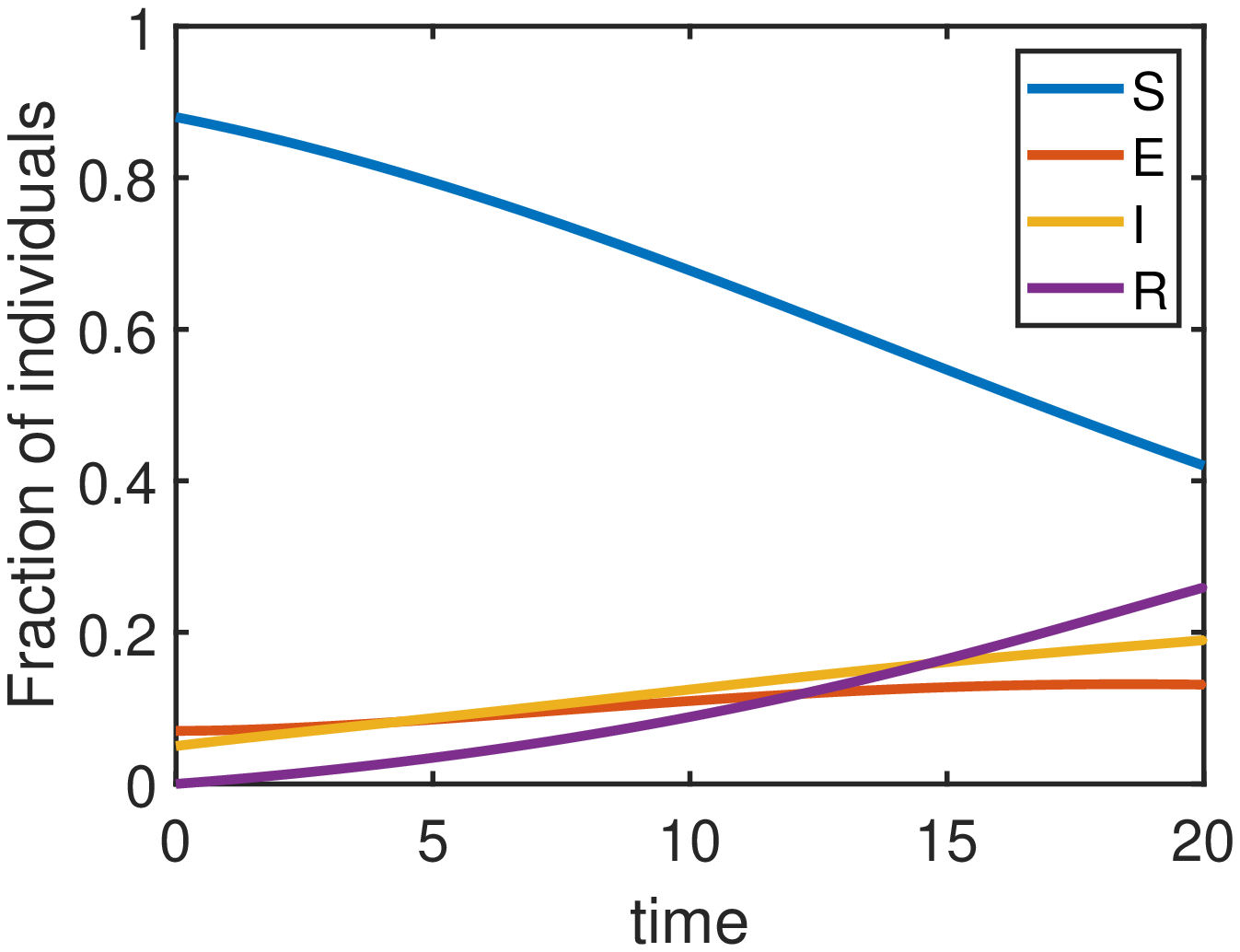}}
\quad
\subfigure[]{\includegraphics[scale=0.3]{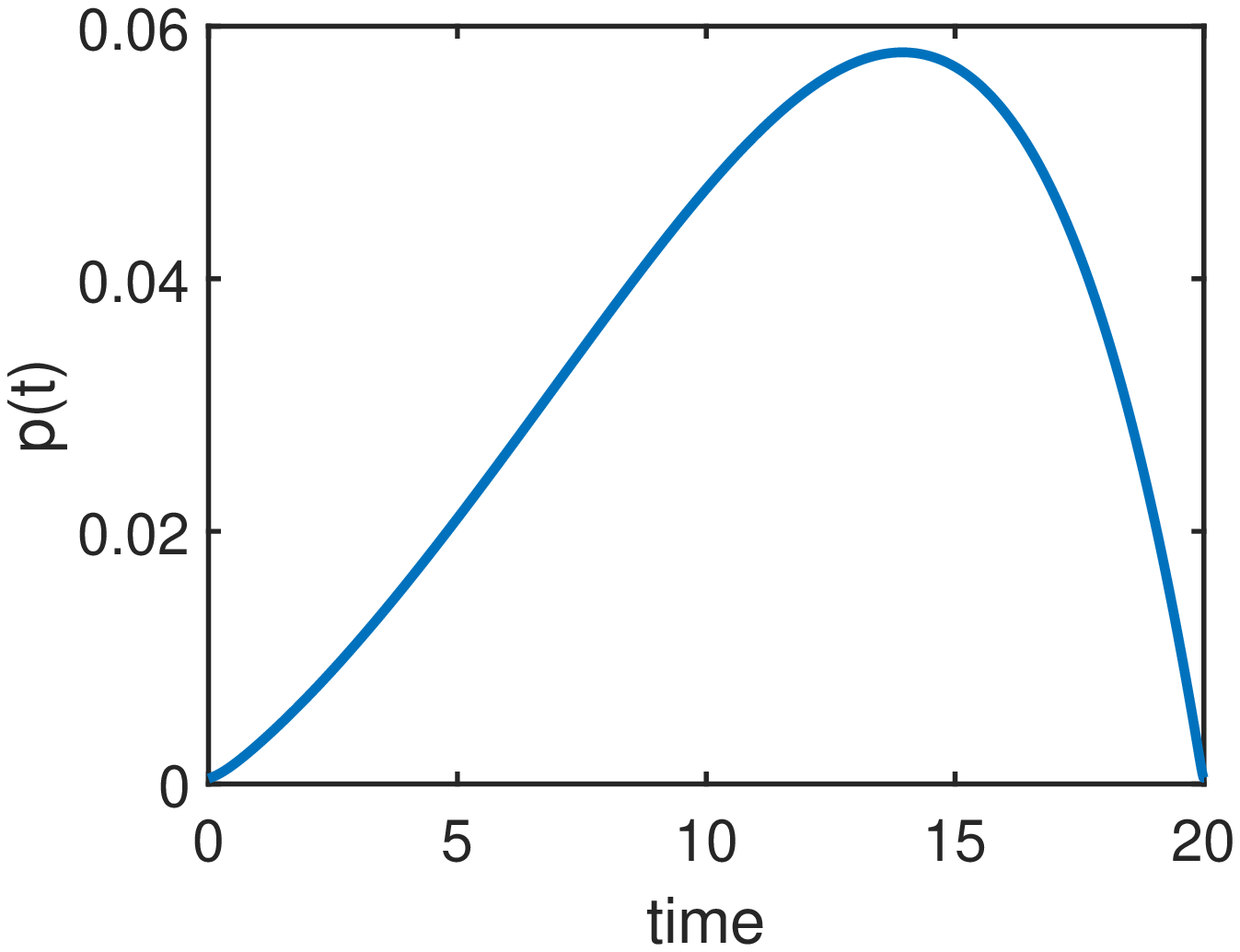}}\\
\caption{Effect of infusing infectious individuals with plasma, 
during 20 time units, considering $(OC_3)$. (a) $seir$ state variables 
applying plasma transfusion. (b) Plasma transfusion control $p(\cdot)$.}
\label{fig:plasma}
\end{center}
\end{figure}
\begin{figure}[ht!]
\begin{center}
\subfigure[]{\includegraphics[scale=0.3]{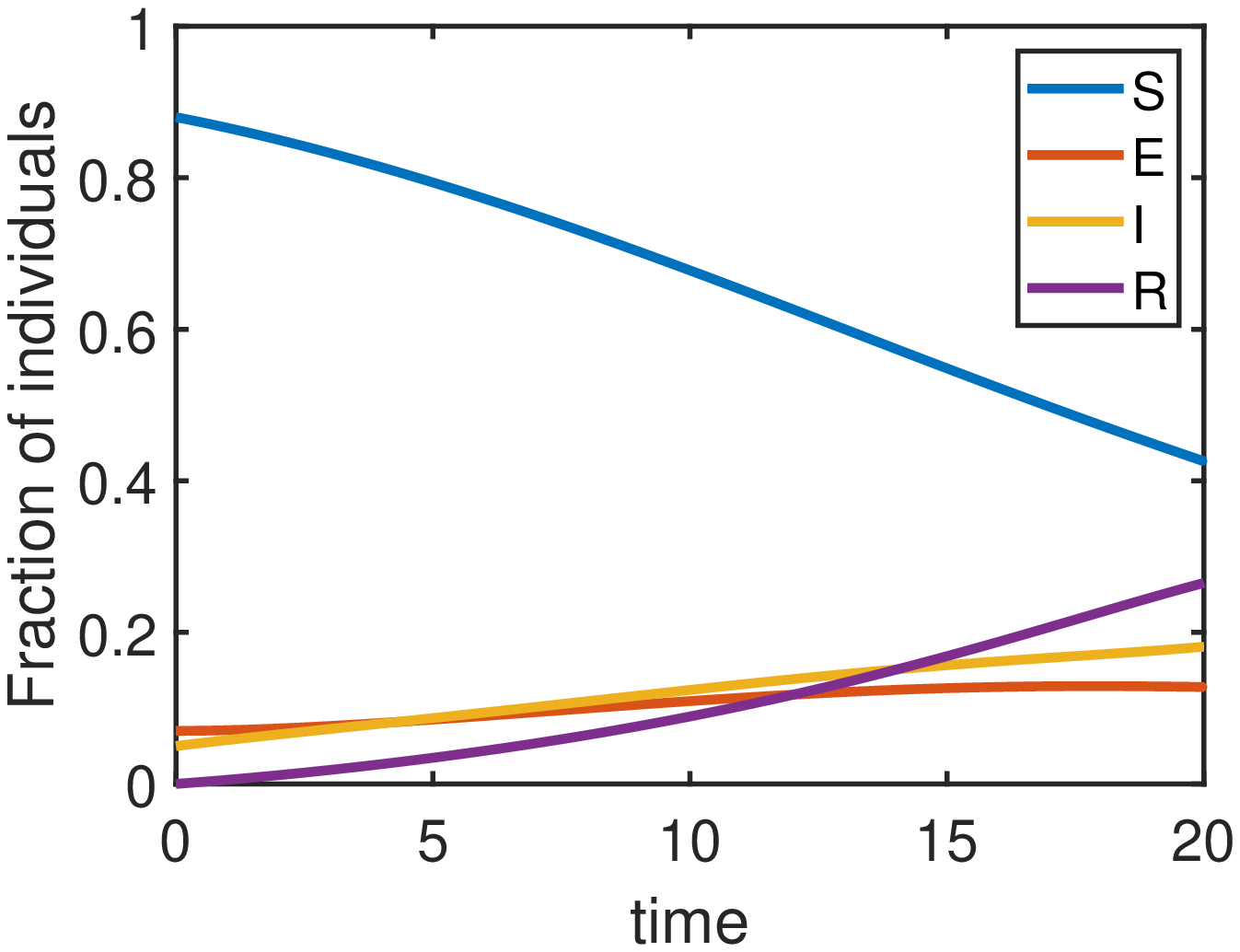}}
\quad
\subfigure[]{\includegraphics[scale=0.3]{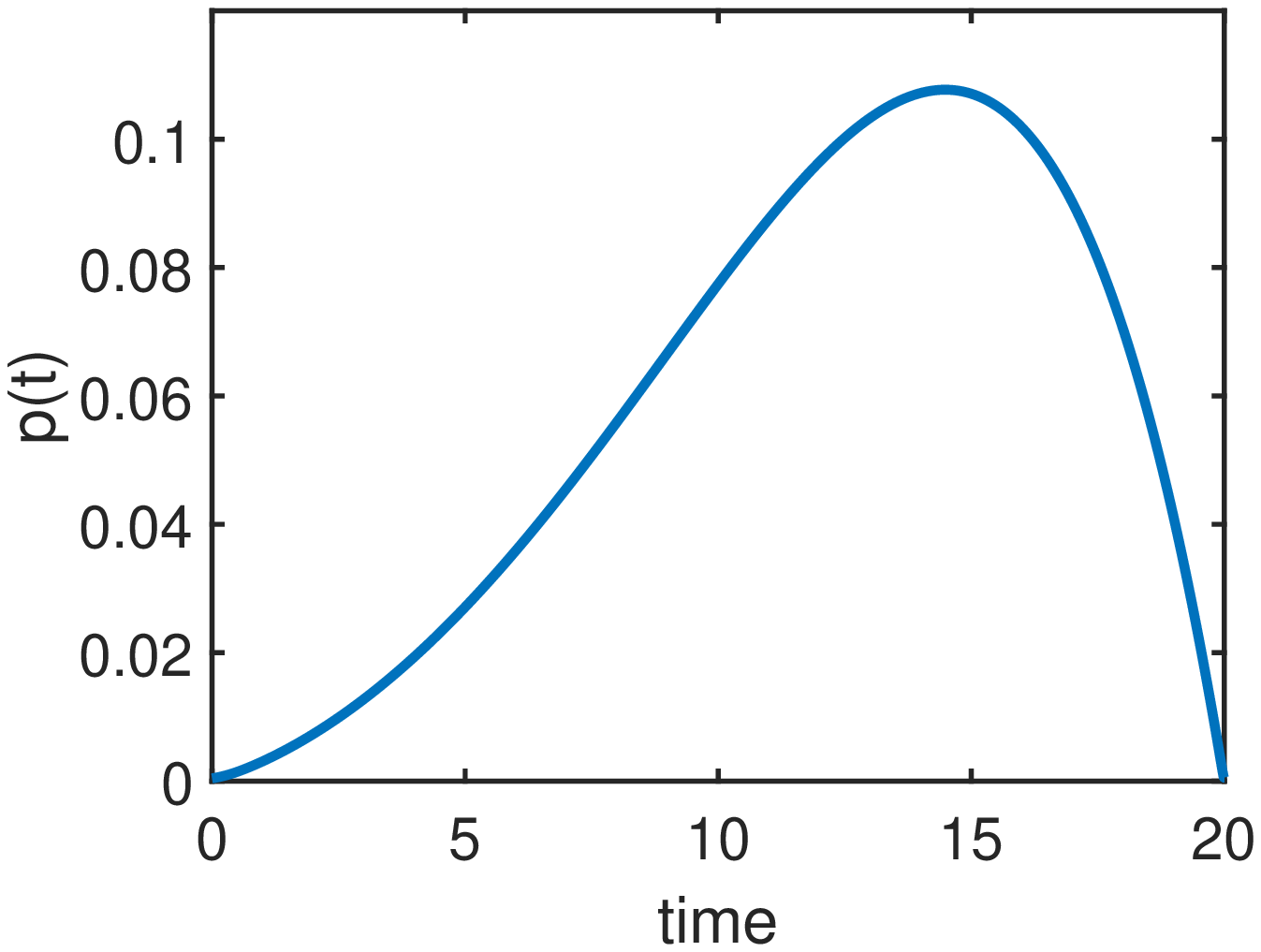}}\\
\caption{Effect of infusing infectious individuals with plasma, 
during 20 time units, considering $(OC_4)$. (a) $seir$ state variables applying 
plasma transfusion. (b) Plasma transfusion control $p(\cdot)$.}
\label{fig:plasmav2}
\end{center}
\end{figure}
Again, this is something that makes sense since, in order for the treatment to be applied, 
there must be not only individuals in the infected $i$ state, that are able to received 
the plasma, but also individuals in the recovered $r$ state, that are able to donate the plasma. 
Evidently, these recovered individuals $r$ must have been in the infected state $i$ before.

Because this is an intervention that presupposes that the disease has evolved for some time, 
then a larger time window could allow one to visualize a stronger impact in the fractions 
of the $i$ and $r$ states. That said, a simulation for optimal control problems 
$(OC_3)$ and $(OC_4)$ is performed using $T = 100$.

The simulation that made the control increase the most with the time change 
was the one for the optimal control problem $(OC_3)$
(Figure~\ref{fig:plasma_100}). Further, the control peak at 
Figure~\ref{fig:plasma_100} also occurs before the control peak at 
Figure~\ref{fig:plasmav2_100}. This is expected since $(OC_4)$ requires 
maximizing the $r$ state, which implies again that more individuals must get 
into the $i$ state so that they can get into the $r$ state after recovering. 
It is also interesting to note that, according to the $(OC_4)$ optimal control problem, 
one should not proceed with plasma transfusion to any infected individual 
during the beginning of the epidemic, in order to obtain the optimal control.
\begin{figure}[ht!]
\begin{center}
\subfigure[]{\includegraphics[scale=0.3]{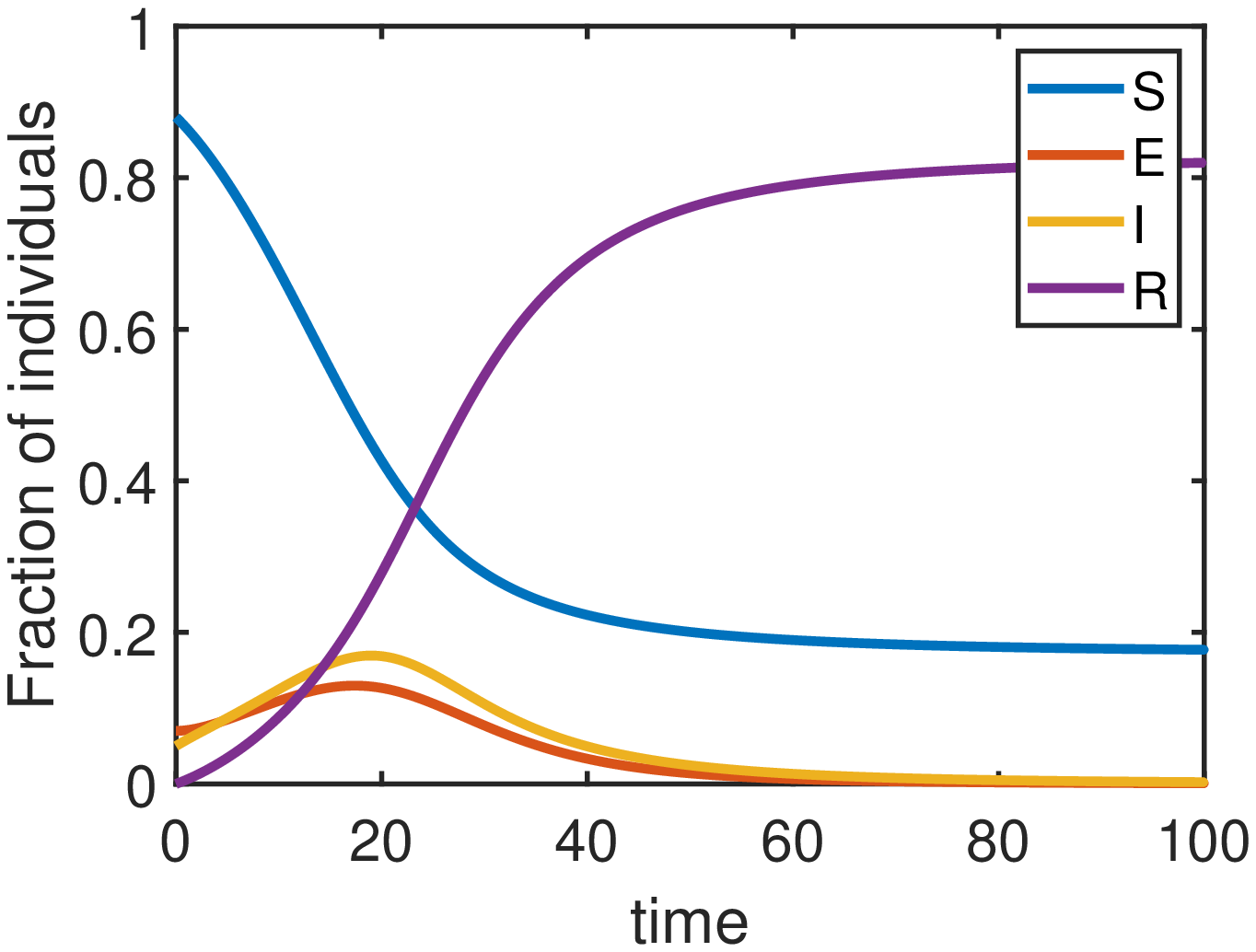}}
\quad
\subfigure[]{\includegraphics[scale=0.3]{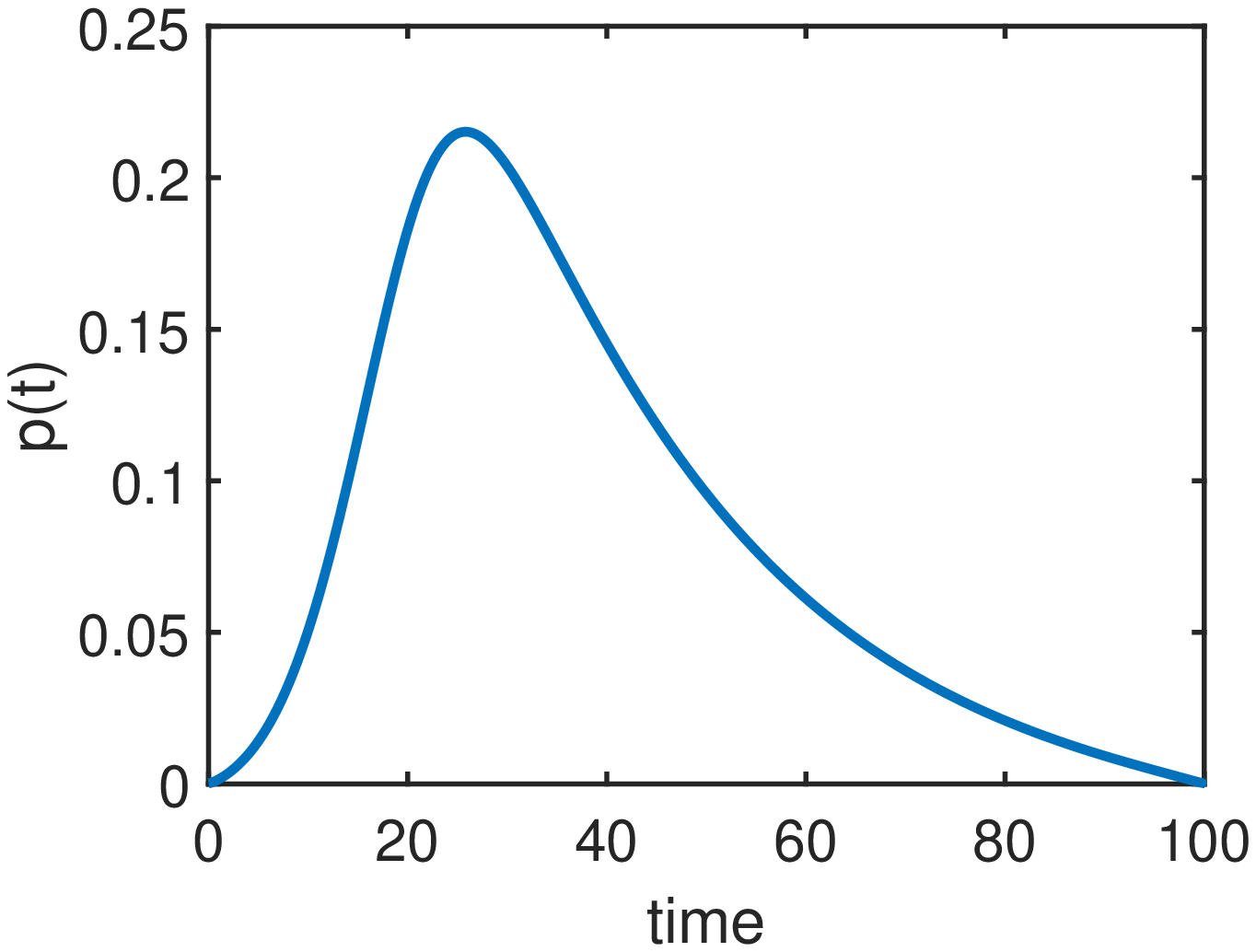}}\\
\caption{Effect of infusing infectious individuals with plasma, 
during 100 time units, considering the cost functional $J_3$. (a) $seir$ state variables, applying 
plasma transfusion. (b) Plasma transfusion control $p(\cdot)$.}
\label{fig:plasma_100}
\end{center}
\end{figure}
\begin{figure}[ht!]
\begin{center}
\subfigure[]{\includegraphics[scale=0.3]{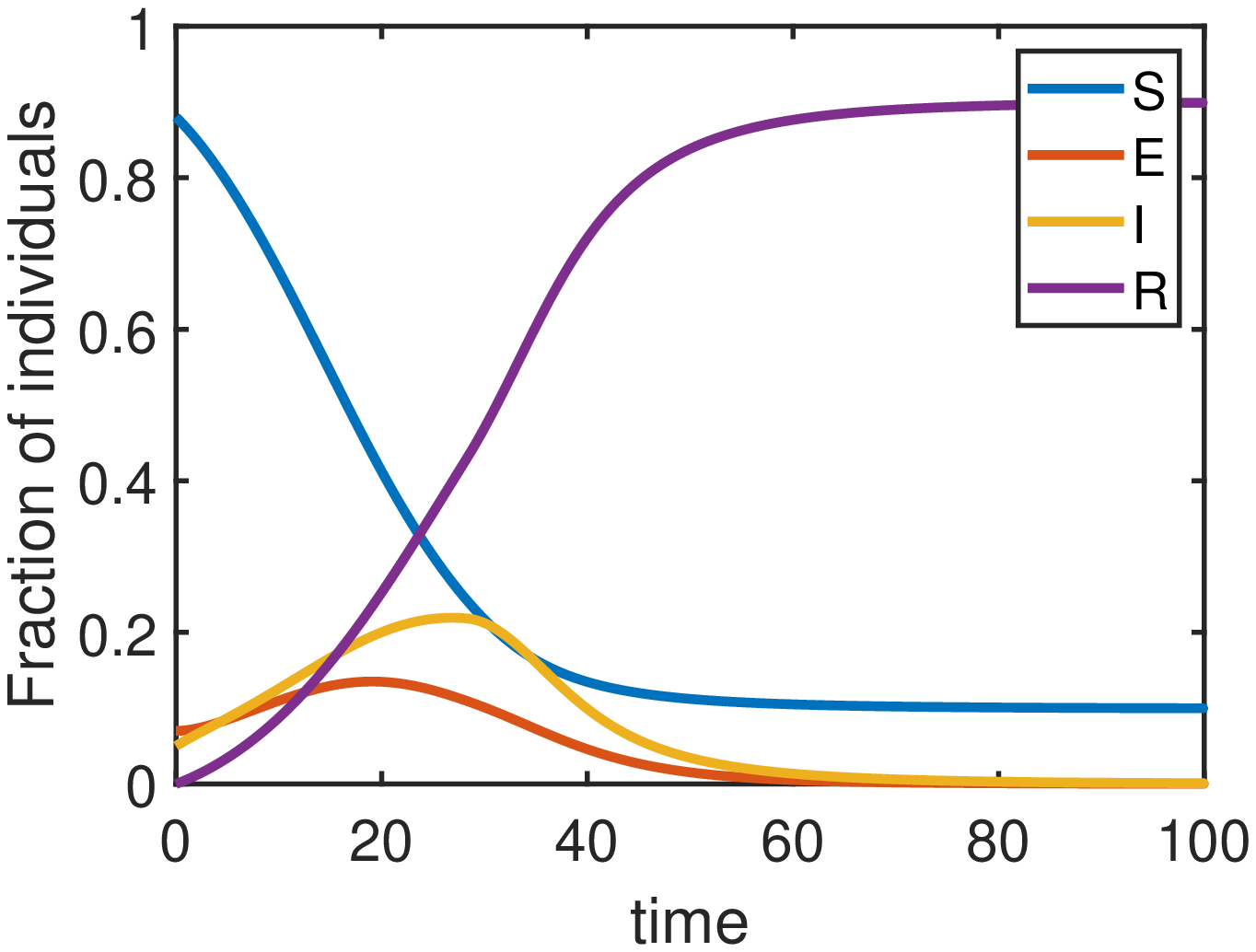}}
\quad
\subfigure[]{\includegraphics[scale=0.3]{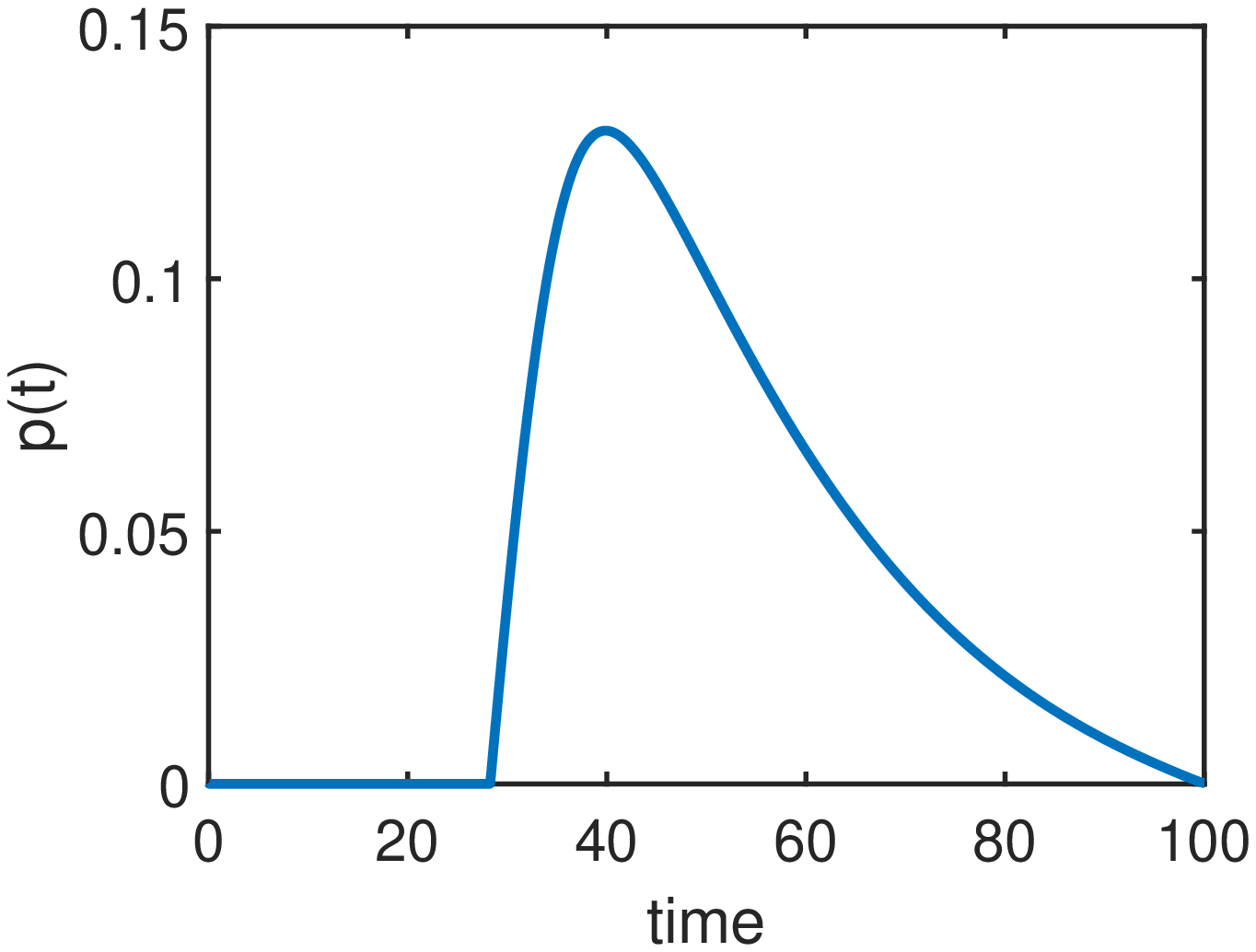}}\\
\caption{Effect of infusing infectious individuals with plasma, 
during 100 time units, considering the cost functional $J_4$. 
(a) $seir$ state variables applying plasma transfusion. 
(b) Plasma transfusion control $p(\cdot)$.}
\label{fig:plasmav2_100}
\end{center}
\end{figure}


\subsubsection{Optimal control problem $(OC_5)$} 

Finally, the combined effect of the two controls for the optimal control 
problem $(OC_5)$ is presented in Figure~\ref{fig:vaccine-plasma}. As expected, 
the peak of the vaccination rate occurs before the peak of the plasma transfusion rate. 
Apparently, the results in minimizing the fraction of individuals in the infected state $i$ 
are better, when comparing Figure~\ref{fig:vaccine-plasma}a 
with Figures~\ref{fig:vaccine}a and \ref{fig:plasma}a, 
but Figure~\ref{fig:individual-seir-control} gives us a better understanding 
of the controls effects in the $seir$ dynamics.  
\begin{figure}[ht!]
\begin{center}
\subfigure[]{\includegraphics[scale=0.3]{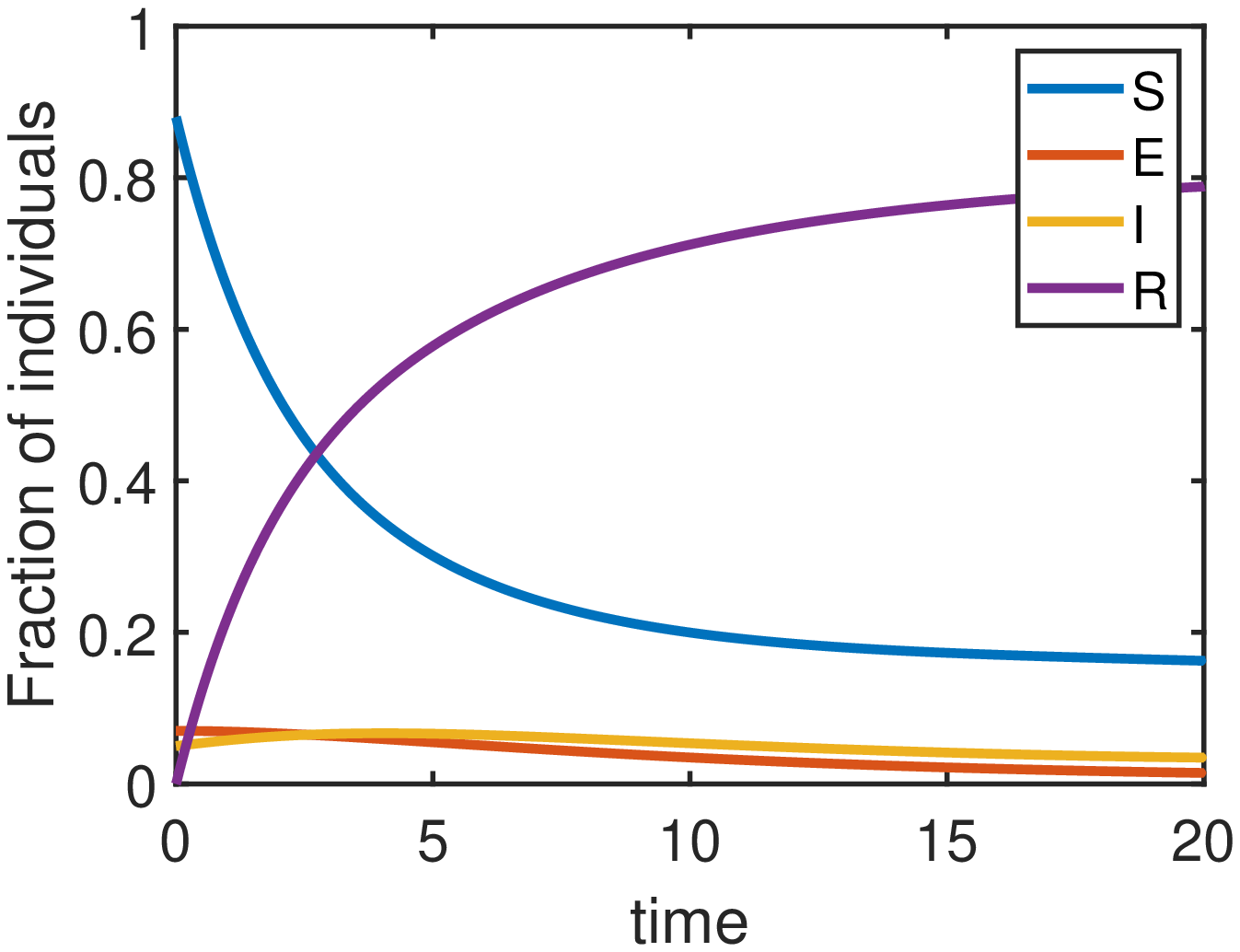}}
\quad
\subfigure[]{\includegraphics[scale=0.3]{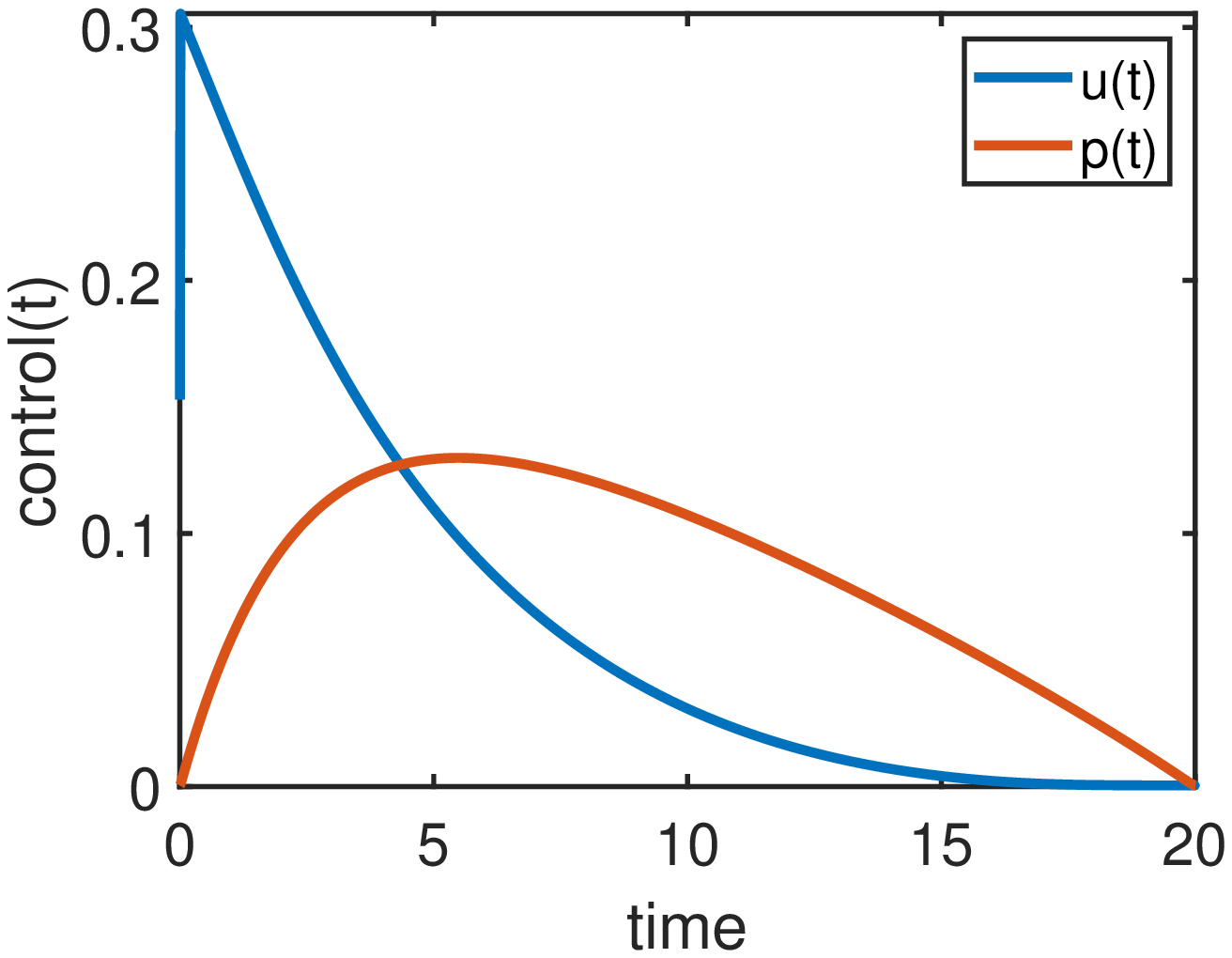}}
\caption{Effect of both vaccinating susceptible individuals and infusing 
infectious individuals with plasma during 20 time units considering $J_5$. 
(a) $seir$ states applying vaccination and plasma transfusion. 
(b) Controls $u(t)$ and $p(t)$.}
\label{fig:vaccine-plasma}
\end{center}
\end{figure}
Figure~\ref{fig:individual-seir-control} shows the effect of the controls 
in the individual $s$, $e$, $i$, $r$ states. Since the main objective 
of the control functionals is to minimize the number of individuals 
in the infected state $i$, then, by looking at Figure~\ref{fig:individual-seir-control}c, 
one can see that the control that minimizes the $i$ fraction the most is the 
conjugation of both vaccination and plasma transfusion. This is also an intuitive result, 
since the vaccination makes more people jumping into the $r$ state, which is the 
pool of individuals from where the plasma comes. On the other hand, 
the plasma transfusion by itself seems to be the less effective control 
in minimizing the infected fraction $i$ (Figure~\ref{fig:individual-seir-control}c). 
This is expected since, as explained above, the plasma transfusion control needs more time 
to kick in the absence of a larger pool of recovered individuals $r$. 

Furthermore, if the aim is to maximize the recovered state $r$ 
or to minimize the susceptible $s$ and exposed $e$ states, then 
the vaccination is the best control (Figure~\ref{fig:individual-seir-control}d), 
a result also predicted by system \eqref{eq:vaccines}.
\begin{figure}[ht!]
\begin{center}
\subfigure[]{\includegraphics[scale=0.3]{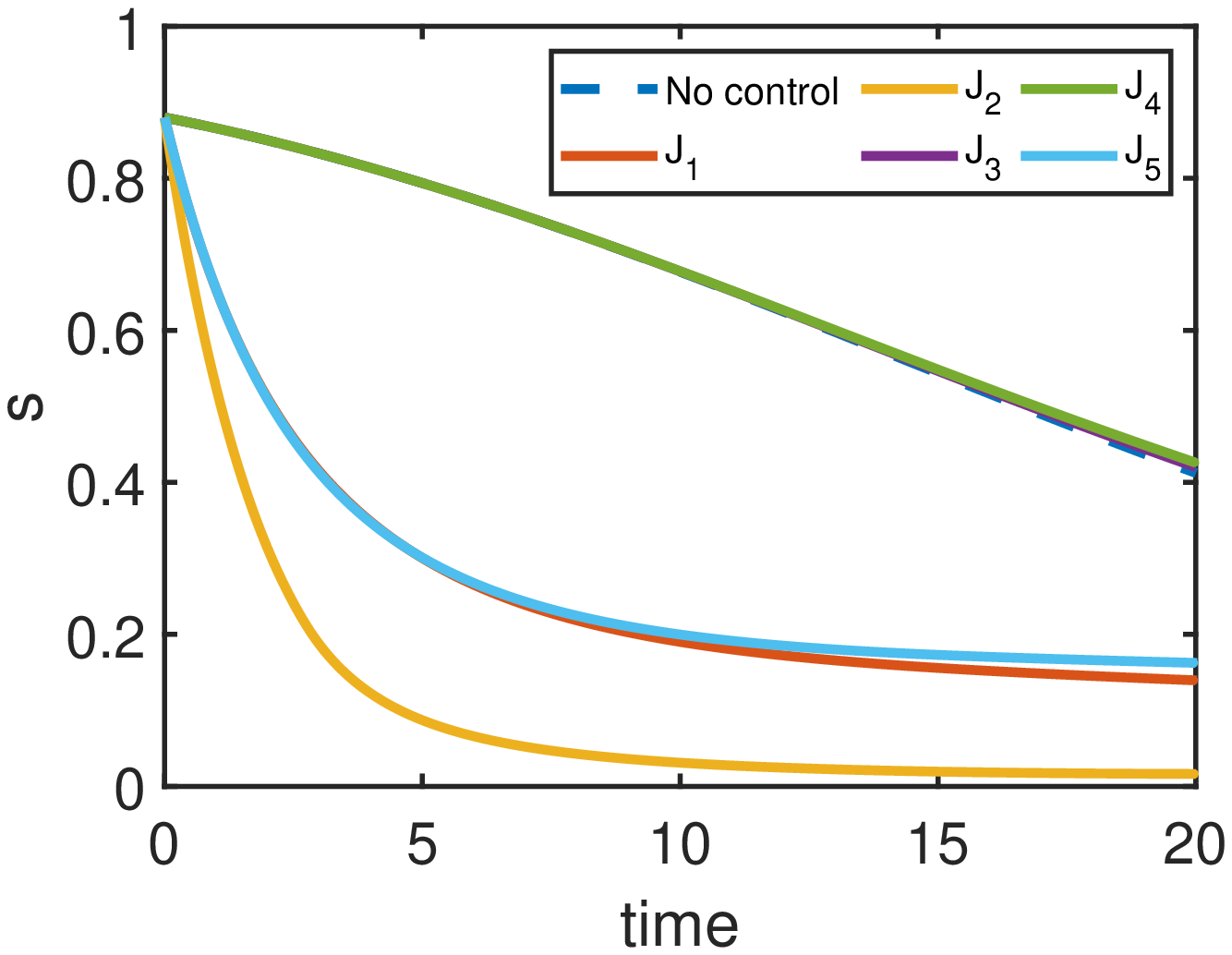}}
\quad
\subfigure[]{\includegraphics[scale=0.3]{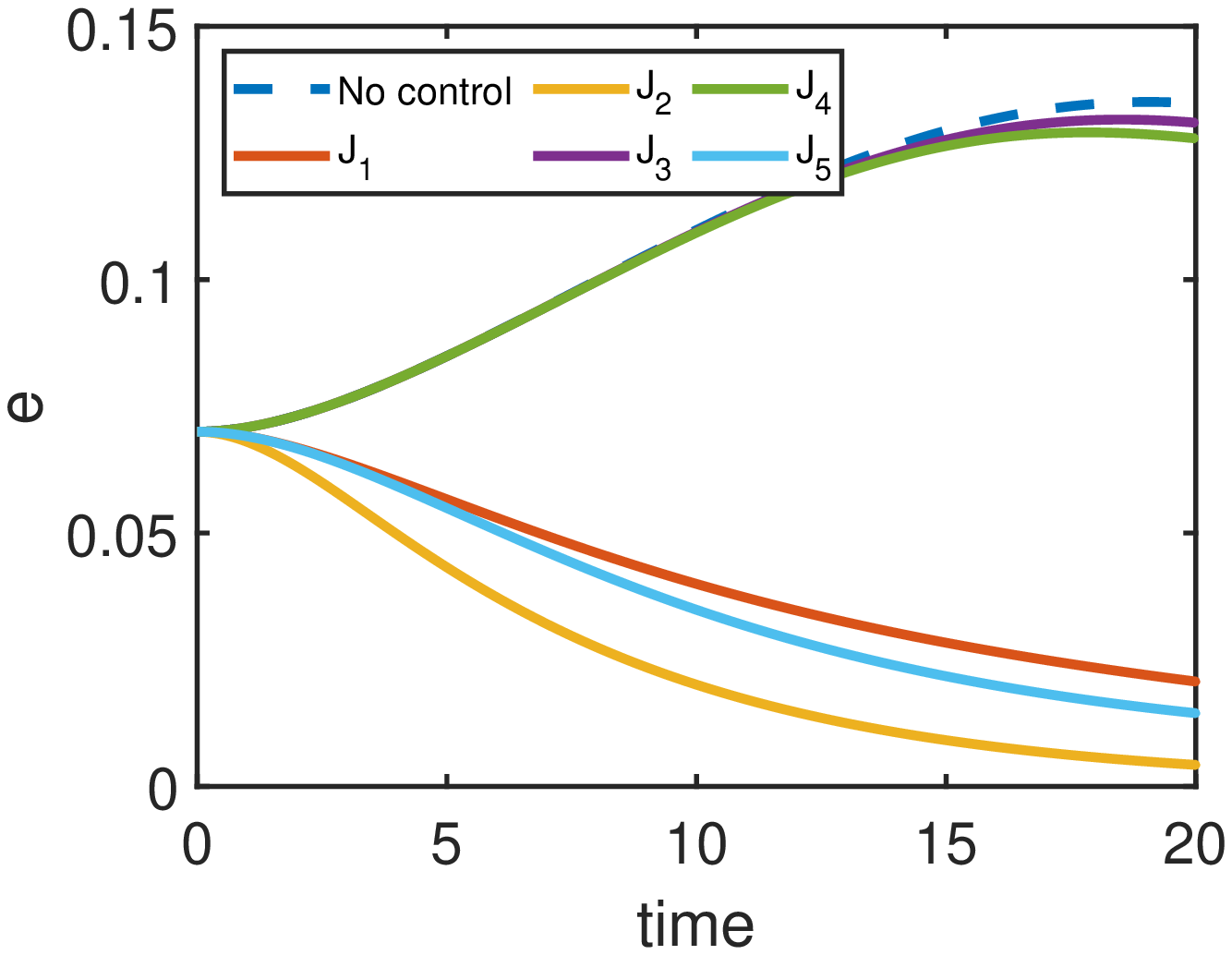}}\\
\subfigure[]{\includegraphics[scale=0.3]{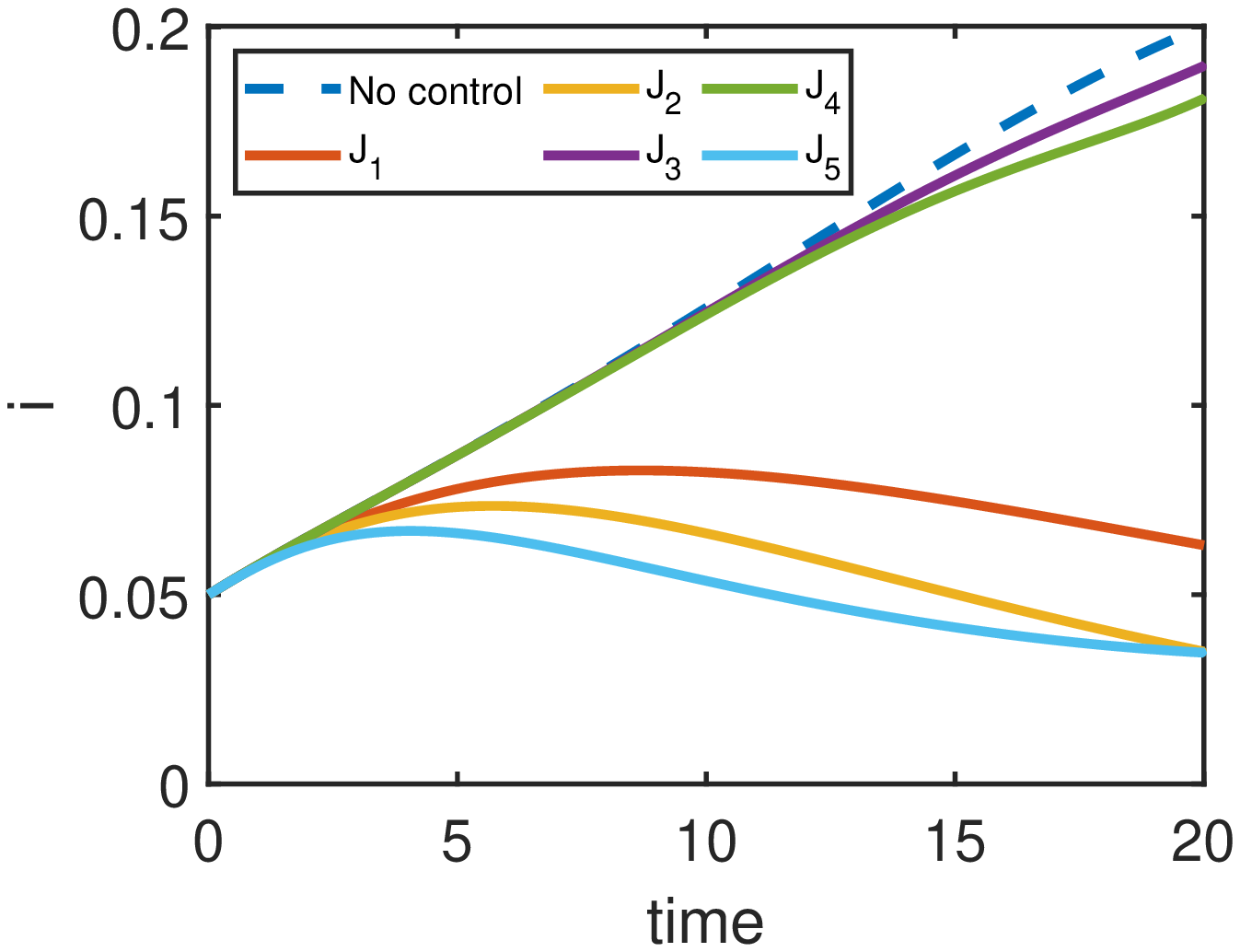}}
\quad
\subfigure[]{\includegraphics[scale=0.3]{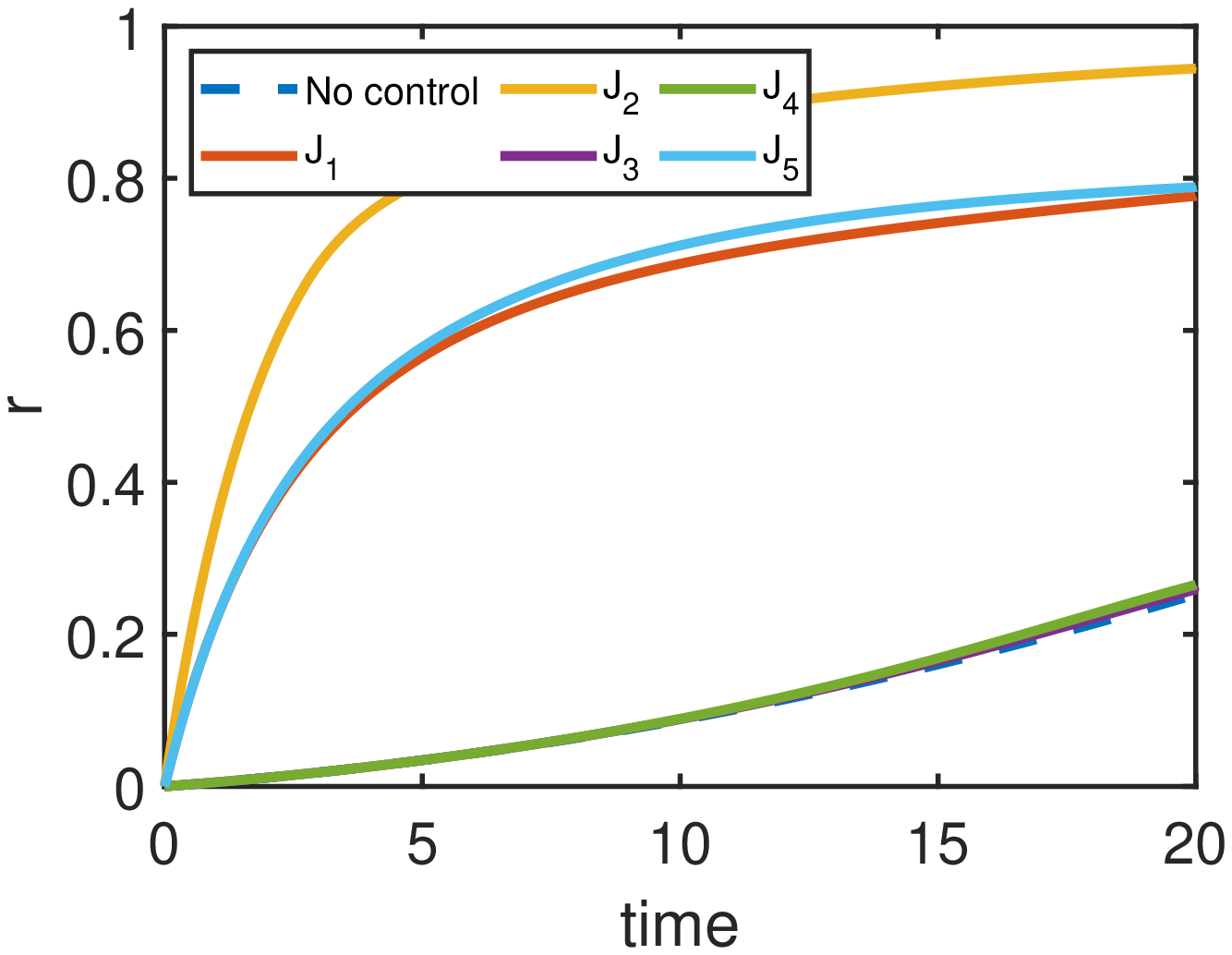}}
\caption{Individual evolution of the controlled and uncontrolled states 
$s$ $e$ $i$ $r$, during 20 time units. (a) susceptible state $s$. (b) exposed state $e$. 
(c) infected state $i$. (d) removed state $r$.}
\label{fig:individual-seir-control}
\end{center}
\end{figure}
By analysing the control comparison at Figure~\ref{fig:control-vaccine-plasma}, 
one can see that the control that benefits the most with the combined approach 
of both vaccination and plasma transfusion is the plasma control. This can
be explained since one can think that if we apply vaccination in the beginning 
of the epidemics, then the fraction of recovered individuals $r$ increases faster, 
providing a bigger substract to do plasma transfusion sooner and at a higher rate.  
\begin{figure}[ht!]
\begin{center}
\subfigure[]{\includegraphics[scale=0.3]{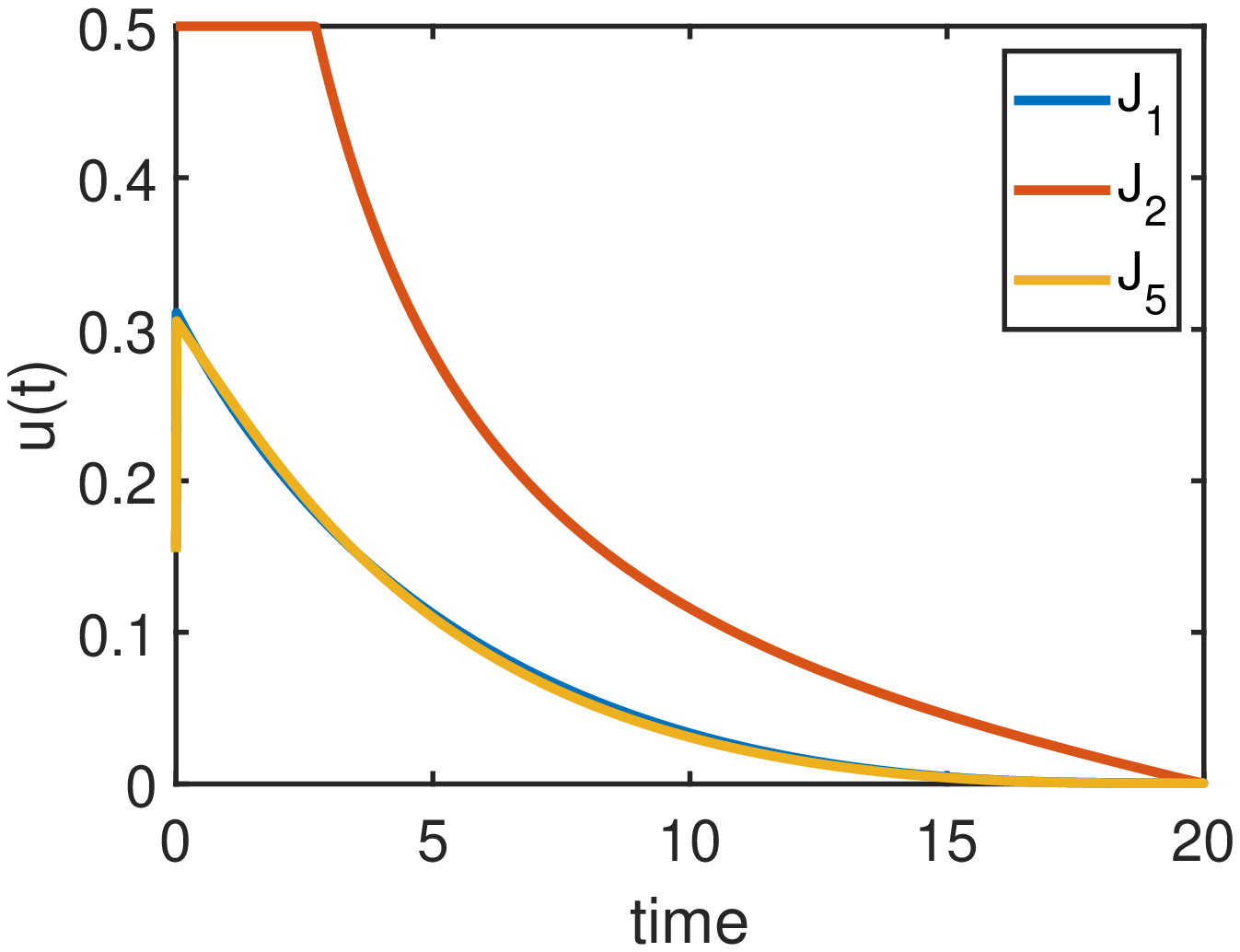}}
\quad
\subfigure[]{\includegraphics[scale=0.3]{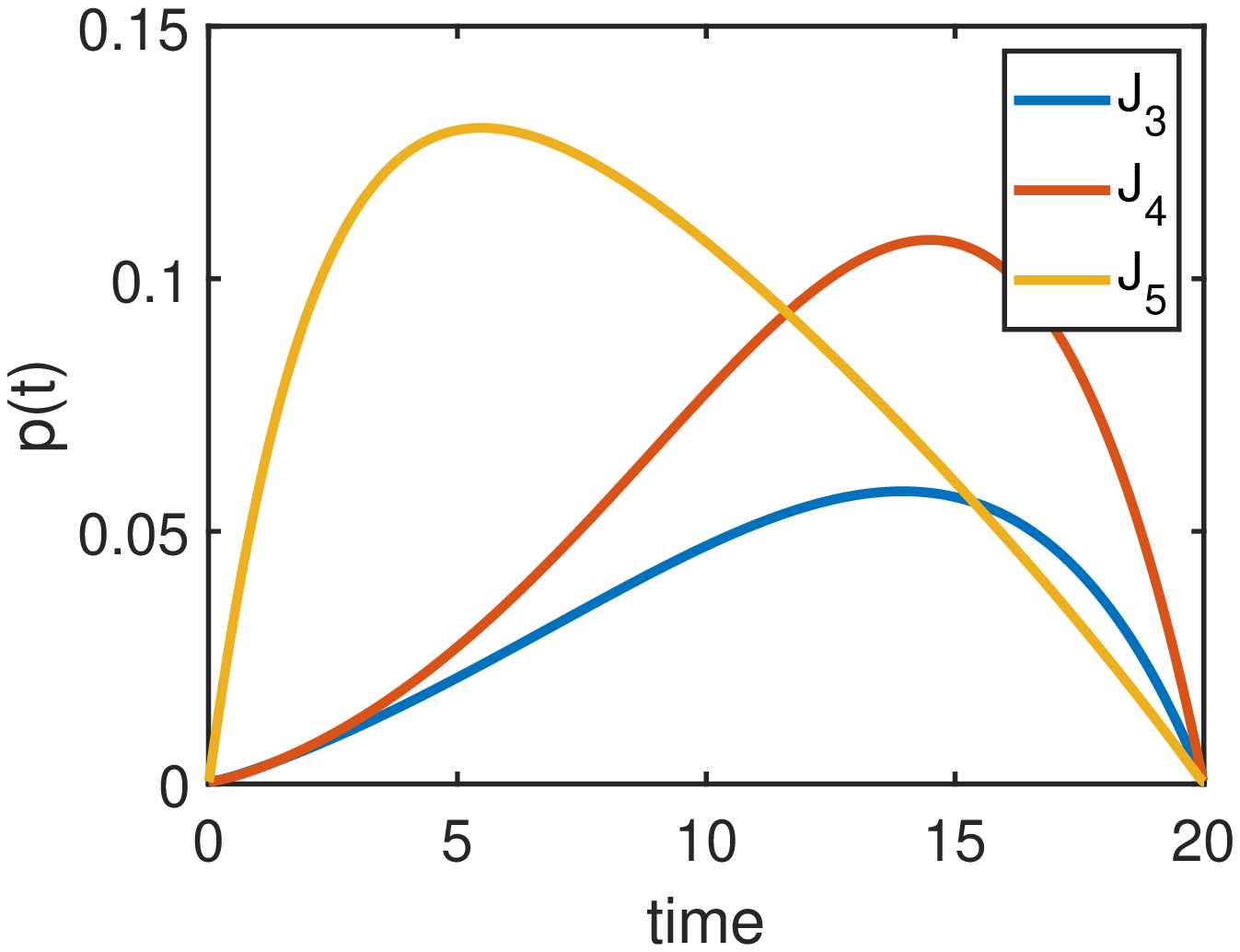}}
\caption{Evolution of vaccination and plasma transfusion controls, 
during 20 time units. (a) Vaccination control $u(\cdot)$. 
(b) Plasma transfusion control $p(\cdot)$.}
\label{fig:control-vaccine-plasma}
\end{center}
\end{figure}

All the optimal control simulations were carried out using NEOS Server 6.0,
their duration varying between 0.508 and 1.149 seconds, for 20 units of time, 
and between 24.431 and 30.702 seconds, for 100 units of time.


\section{Discussion and conclusion}
\label{sec:conc}

The $seir$ model \eqref{eq:mod:SEIR} was solved numerically in both uncontrolled 
and controlled conditions. Of the controls employed, the combined action of 
vaccination and plasma transfusion seems to have the higher impact in reducing 
the fraction of infectious individuals. Moreover, the plasma transfusion acquires 
a more important role as the fraction of individuals in the recovered state increases, 
which explains why whether joining the vaccination or increasing the duration 
of the simulations leads to an higher peak of the plasma transfusion rate. 
In fact, by joining the vaccination, one can not only increase the plasma transfusion 
rate peak but also anticipate it. 

To sum up, controls can act at different timings of the epidemics dynamics and one 
control can be more adequate in the beginning of the epidemic whilst 
other might be more appropriated in a later state. 


\section*{Code availability}

The code is available from the authors on request. 


\begin{acknowledgement}
This research was partially supported by the Portuguese Foundation 
for Science and Technology (FCT) within ``Project n.~147 -- Controlo 
\'Otimo e Modela\c{c}\~ao Matem\'atica da Pandemia \text{COVID-19}: 
contributos para uma estrat\'egia sist\'emica de interven\c{c}\~ao 
em sa\'ude na comunidade'', in the scope of the ``RESEARCH 4 COVID-19'' 
call financed by FCT. The work of Silva and Torres was also partially 
supported within project UIDB/04106/2020 (CIDMA). Moreover, Silva 
is also supported by national funds (OE), through FCT, I.P., in the scope 
of the framework contract foreseen in the numbers 
4, 5 and 6 of the article 23, of the Decree-Law 57/2016,
of August 29, changed by Law 57/2017, of July 19.
The authors are grateful to two anonymous reviewers for helpful
comments and suggestions.
\end{acknowledgement}




\begin{thebibliography}{99}

\bibitem{Plasma:2}
American Society of Hematology,
COVID-19 and Convalescent Plasma: Frequently Asked Questions, 
\url{https://www.hematology.org/covid-19/covid-19-and-convalescent-plasma}

\bibitem{AreaEbola}
Area, I., Ndairou, F., Nieto, J.J., Silva, C.J. and Torres,  D.F.M.:
Ebola Model and Optimal Control with Vaccination Constraints,
J. Ind. Manag. Optim. \textbf{14}, no. 2, 427--446 (2018).
{\tt arXiv:1703.01368}

\bibitem{brauer2011mathematical} 
Brauer, F. and Castillo-Chavez, C.:
Mathematical Models in Population Biology and Epidemiology,
Texts in Applied Mathematics, Springer New York (2011)

\bibitem{CarlosCampos}
Campos, C., Silva, C.J. and Torres, D.F.M.: 
Numerical Optimal Control of HIV Transmission in Octave/MATLAB, 
Math. Comput. Appl. \textbf{25}, no.~1, 20~pp (2020).
{\tt arXiv:1912.09510}

\bibitem{Carcione:Lombardia}
Carcione, J.M., Santos, J.E., Bagaini, C. and Ba, J.: 
A Simulation of a COVID-19 Epidemic Based on a Deterministic SEIR Model, 
Front. Public Health 8:230 (2020). 

\bibitem{Cesari}
Cesari, L. 
Optimization -- Theory and Applications. 
Problems with Ordinary Differential Equations, 
Applications of Mathematics 17, Springer-Verlag, New York, 1983.

\bibitem{covid19projectionsML}
COVID-19 Projections Using Machine Learning, 
\url{https://covid19-projections.com}

\bibitem{AMPL}
Fourer, R., Gay, D.M. and Kernighan, B.W.: 
AMPL: A Modeling Language for Mathematical
Programming, Duxbury Press, Brooks–Cole Publishing Company, 1993.

\bibitem{Jung:2002}
Jung, E., Lenhart, S., Feng, Z.: 
Optimal control of treatments in a two-strain tuberculosis model, 
Discrete Contin. Dyn. Syst. Ser. B \textbf{2} (4), 473--482 (2002).

\bibitem{KermackMcKendrick:1927}
Kermack, W.O. and McKendrick, A.G.: 
A Contribution to the Mathematical Theory of Epidemics,  
Proc. Roy. Soc. Lond. A \textbf{115}, 700-7-21 (1927). 

\bibitem{KermackMcKendrick:1991}
Kermack, W.O., McKendrick, A.G. 
Contributions to the mathematical theory of epidemics I. 
Bltn Mathcal Biology \textbf{53}, 33--55 (1991). 

\bibitem{LemosPaiao:cholera}
Lemos-Pai\~ao, A.P., Silva, C.J., Torres, D.F.M. and Venturino, E.:
Optimal Control of Aquatic Diseases: A Case Study of Yemen's Cholera Outbreak,
J. Optim. Theory Appl. \textbf{185}, no. 3, 1008--1030 (2020).
{\tt arXiv:2004.07402}

\bibitem{Lopez:SEIR:Spain:Italy}
L\'opez, L. and Rod\'o, X.: 
A Modified SEIR Model to Predict the COVID-19 Outbreak in Spain and Italy: 
Simulating Control Scenarios and Multi-Scale Epidemics, 
Available at SSRN: \url{http://dx.doi.org/10.2139/ssrn.3576802}

\bibitem{Murray:book} 
Murray, J.D.: 
Mathematical Biology, Springer Berlin Heidelberg (2013) 

\bibitem{[16]} 
Nemati, S.; Lima, P.M.; Torres, D.F.M. 
A numerical approach for solving fractional optimal control
problems using modified hat functions. 
Commun. Nonlinear Sci. Numer. Simul. 78 (2019), Art.~104849, 14~pp.
{\tt arXiv:1905.06839}

\bibitem{neos}
NEOS Interfaces to Ipopt,
\url{https://neos-server.org/neos/solvers/nco:Ipopt/AMPL.html}.

\bibitem{Ng:COVID19}
Ng, K. Y. and Gui, M. M.: 
COVID-19: Development of a robust mathematical model and simulation package 
with consideration for ageing population and time delay for control action 
and resusceptibility. Physica D. Nonlinear phenomena, \textbf{411}, 132599 (2020). 

\bibitem{Pontryagin}
Pontryagin, L., Boltyanskii, V., Gramkrelidze, R. and Mischenko, E. 
The Mathematical Theory of Optimal Processes, Wiley Interscience, 1962.

\bibitem{Prem:covid19:model}
Prem, K. et. al.:  
The effect of control strategies to reduce social mixing on outcomes 
of the COVID-19 epidemic in Wuhan, China: a modelling study, 
The Lancet Public Health, \textbf{5} (5), e261--e270 (2020). 
	
\bibitem{[15]} 
Salati, A.B.; Shamsi, M.; Torres, D.F.M. 
Direct transcription methods based on fractional integral
approximation formulas for solving nonlinear fractional optimal control problems. 
Commun. Nonlinear Sci. Numer. Simul. 67 (2019) 334--350.
{\tt arXiv:1805.06537}

\bibitem{Schattler:book}
Sch\"attler, H. and Ledzewicz, U., 
Optimal Control for Mathematical Models of Cancer Therapies, An Application of Geometric Methods,
Springer-Verlag New York, 2015. 

\bibitem{SEIR:comsol}
SEIR Model for the COVID-19 Epidemic, 
\url{https://www.comsol.pt/model/seir-model-for-the-covid-19-epidemic-86511}. 

\bibitem{matlab:ode}
Shampine, L.F. and  Reichelt, M. W.: The MATLAB ODE Suite, 
SIAM Journal on Scientific Computing \textbf{18}, 1--22 (1997).

\bibitem{HIVSilvaMaurer}
Silva, C.J. and Maurer, H.:
Optimal control of HIV treatment and immunotherapy combination with state and control delays,
Optim Control Appl. Meth. \textbf{41}, 537--554 (2020). 

\bibitem{MaurerSilvaTorresMBE}
Silva, C.J., Maurer, H. and Torres, D.F.M.: 
Optimal control of a tuberculosis model with state and control delays, 
Math. Biosci. Eng. \textbf{14}, no. 1, 321--337 (2017). 
{\tt arXiv:1606.08721}

\bibitem{SilvaTorres:2013}
Silva, C.J., Torres, D.F.M.: 
Optimal control for a tuberculosis model with reinfection and post-exposure interventions, 
Math. Biosci. \textbf{244}, no. 2, 154--164 (2013).
{\tt arXiv:1305.2145}

\bibitem{TBHIVSilvaTorres}
Silva, C.J. and Torres, D. F. M.:
A TB-HIV/AIDS coinfection model and optimal control treatment,
Discrete Contin. Dyn. Syst., \textbf{35}, no. 9, 4639--4663 (2015).
{\tt arXiv:1501.03322}

\bibitem{SICA:2020} 
Silva, C. J. and Torres, D. F. M.: 
On SICA models for HIV transmission.
In: \textit{Mathematical Modelling and Analysis of Infectious Diseases}, 
ed. by K. Hattaf and H. Dutta., 
Studies in Systems, Decision and Control 302 (2020),
Springer Nature Switzerland, 2020, 155--179.
{\tt arXiv:2004.11903}
	
\bibitem{IPOPT}
W\"achter, A. and Biegler, L. T.: 
On the implementation of an interior-point filter line-search
algorithm for large-scale nonlinear programming, 
Math. Program. \textbf{106}, 25--57 (2006).
	
\end{thebibliography}
\end{document}